\documentclass[12pt]{article}

\usepackage[english]{babel}
\usepackage{authblk}
\usepackage[utf8]{inputenc}
\usepackage[T1]{fontenc}

\usepackage{enumitem}
\usepackage{nccmath}
\usepackage{amsthm} 
\usepackage{amsmath}
\usepackage{mathrsfs}
\usepackage{amssymb}
\usepackage{mathabx}

\usepackage{fancyhdr}
\usepackage{epsfig}
\usepackage{epic}
\usepackage{eepic}
\usepackage{threeparttable}
\usepackage{amscd}
\usepackage{here}
\usepackage{graphicx}
\usepackage{lscape}
\usepackage{tabularx}
\usepackage{subfigure}
\usepackage{longtable}
\usepackage{makeidx}
\usepackage{tikz-cd}
\usepackage[T1]{fontenc}
\usepackage{textcomp}
\usepackage{thmtools}
\usepackage{hyperref}
\usepackage{babelbib}
\usepackage{setspace}
\usepackage{marginnote}
\usepackage{multicol}
\usepackage{authblk}
\usepackage{graphicx}

\makeindex 
\newcommand\scalemath[2]{\scalebox{#1}{\mbox{\ensuremath{\displaystyle #2}}}}

\theoremstyle{definition}
\newtheorem{Teo}{Theorem}[section]

\newtheorem{Coro}[Teo]{Corollary}

\newtheorem{Obs}[Teo]{Remark}
\newtheorem{Lema}[Teo]{Lemma}
\newtheorem{Def}[Teo]{Definition}
\newtheorem{Ej}[Teo]{Example}

\begin{document}

\title{Kripke-like models of Set Theory in Modal Residuated Logic.}

\author{Jose Moncayo\thanks{jrmoncayov@unal.edu.co} }
\author{Pedro H. Zambrano\thanks{phzambranor@unal.edu.co}}
\affil{Departamento de Matem\'aticas, Universidad Nacional de Colombia, AK 30 $\#$ 45-03 c\'odigo postal 111321, Bogota, Colombia.}
\date{\today}
\maketitle

\begin{abstract}
    We generalize Fitting's work on Intuitionistic Kripke models of Set Theory using Ono and Komori's Residuated Kripke models. Based on these models, we provide a generalization of the von Neumann hierarchy in the context of Modal Residuated Logic and prove a translation of formulas between it and a suited Heyting valued model. We also propose a notion of universe of constructible sets in Modal Residuated Logic and discuss some aspects of it.

\textbf{\small Keywords: Valued models, abstract logics, residuated lattices, Kripke models, constructible sets}\\

\textbf{\small MSC 2020 classification:} 03E70, 03B47, 03B60, 03B45, 06F07.

\end{abstract}

\section{Introduction}

When we tried to generalize the notion of a definable set (see~\cite{MoncayoZambrano2X2}) to the context of quantale-valued models, we found that the resulting classes of constructible sets are also \textbf{two valued}, and therefore are not suitable to study Residuated Logic. Hence, we instead focus on developing the notion of constructible sets in the realm of \textbf{Kripke models}, where these kind of problems are avoided. 

As it turns out, there is a precedent to the idea of considering constructible sets over Kripke models: Fitting \cite{Fitting1969} constructed several models of \textbf{Intuitionistic Set Theory} generalizing both the universes of von Neumann and of Gödel using \textbf{Intuitionistic Kripke models} and then he went on to show how these models can be used to obtain \textbf{classical} proofs of independence in Set Theory.

Thus, since we would like to generalize Fitting's Intuitionistic Kripke models of Set Theory, we would need first a notion of Kripke models for Residuated Set Theory. Ono and Komori \cite{OnoKomori} introduced the notion of semantics for substructural logics without contraction and exchange. These models, that we call \textbf{Residuated  Kripke models} (shortly, \textit{$R$-Kripke models}), generalize the notion of Intuitionistic Kripke models and then, following the ideas of Lano \cite{Lano1992a}, we further generalize these models to \textbf{Modal Residuated Kripke models} (shortly, \textit{$MR$-Kripke models}). The definition for the interpretation of the modality in our definition is original, and it allows a suitable translation between MR-Kripke models and Heyting-valued models. 

We define the model $\mathcal{V}^{\mathbb{P}^*}$, that generalizes the von Neumann hierarchy for Modal Residuated Logic, and we prove that there is a \textbf{Gödel–McKinsey–Tarski-like translation} between this model and a suited Heyting valued model $R^{\mathbb{H}}$. This translation is obtained by first constructing an ``isomorphism'' between $\mathcal{V}^{\mathbb{P}^*}$ and $R^{\mathbb{H}}$ and then proving how this result implies that if $\varphi$ is an $\mathcal{L}_{\in}$-sentence that is valid in $R^{\mathbb{H}}$, then $\Diamond \varphi$ is valid in $\mathcal{V}^{\mathbb{P}^*}$.


\subsection{Commutative integral Quantales}\label{completeresiduatedlattice}

Structures like Quantales (i.e. ordered monoids with a product that distributes over arbitrary supremums) have been studied at least since Ward and Dilworth's work on residuated lattices \cite{WardDilworth1938, Dilworth1939, Ward1938}, where their motivations were more algebraic, since they were studying the lattice of ideals in a ring: Given a ring $R$, the set of ideals of $R$, denoted as $Id(R)$, forms a complete lattice defining infimum and supremum as the intersection and sum of ideals, respectively. The monoid operation $\cdot$ on this lattice is given by multiplication of ideals, and the element $R$ in $Id(R)$ is the identity of this operation.

 But it was not until the work of Mulvey \cite{Mulvey1986}, where the term \textbf{Quantale} was coined as a combination of \textbf{``quantum''} and \textbf{``locale''} and proposed their use for studying \textbf{Quantum Logic} and non-commutative $C^*$-algebras.

 Our motivation for the study of Quantales is somewhat different. We are not interested in Quantales that are non-commutative - as was the case for Mulvey - but rather Quantales that are \textbf{not necessary idempotent}. We are interested in studying Quantales since they semantically capture both Intuitionistic and Fuzzy Logic, so we will focus on the study of \textbf{commutative integral Quantales}. This kind of structures is widely use in the field of \textbf{substructural logics} as semantical counterparts for those logics, particularly Residuated Logic.

\begin{Def}[Commutative integral Quantales]\label{residuatedlattice}

We say that $\mathbb{Q}=(\mathbb{Q},\land,\lor,\cdot,\rightarrow, 1, 0)$ is a \textit{commutative integral Quantale} (or equivalently a \textit{complete residuated lattice}) if:

\begin{enumerate}
    \item $(\mathbb{Q},\land,\lor, 1, 0)$ is a complete bounded lattice with 1 as maximum and 0 as minimum.
     \item $(\mathbb{Q}, \cdot, 1)$ is a commutative monoid.
     \item For all $x, y_i\in \mathbb{Q}$ with $i\in I$,
    \begin{center}
        $x\cdot\bigvee\limits_{i\in I}y_i=\bigvee\limits_{i\in I}(x\cdot y_i)$  
    \end{center}
    and $\rightarrow$ can be defined as $x\rightarrow y:=\bigvee\{z\in\mathbb{Q}: x\cdot z\leq y\}$ 
\end{enumerate}
\end{Def}

\begin{Def}
    We say that $\mathbb{Q}$ is \textit{idempotent} if $x\cdot x=x$ for all $x\in\mathbb{Q}$.
\end{Def}

\begin{Teo}[\cite{BuPi2014}, p. 2]\label{cuantales1}
    Let $\mathbb{Q}$ be a commutative integral Quantale and $x, y, z \in \mathbb{Q}$ with $i\in I$. Then:
    \begin{enumerate}
        \item $x\leq y$ if and only if $(x\rightarrow y)=1$.
        \item $x\cdot(x\rightarrow y)\leq y$.
        \item $(1\rightarrow y)=y$
        \item $0=x\cdot 0=0\cdot x$.
        \item $(0\rightarrow y)=1$ 
        \item If $x\leq y$, then $x\cdot z\leq y\cdot z$.
        \item  $x\cdot y\leq x\land y$.
        \item If $x\leq y$, then $y\rightarrow z\leq x\rightarrow z$.
        \item If $x\leq y$, then $z\rightarrow x\leq z\rightarrow y$.
        \item $(x\cdot y)\rightarrow z= x\rightarrow(y\rightarrow z)$.
    \end{enumerate}
\end{Teo}

\begin{Def}\label{otheroperationsonQ}
    Let $\mathbb{Q}$ be a commutative integral Quantale and $x, y\in \mathbb{Q}$. We define
    \begin{enumerate}
        \item $\sim x:=x\rightarrow 0$ (negation),
        \item $x\equiv y:= (x\rightarrow y)\cdot(y\rightarrow x)$ (equivalence),
        \item $x^0=1$ and $x^{n+1}=x\cdot x^n$ for $x\in\mathbb{N}$ (exponentiation).
    \end{enumerate}
\end{Def}

\begin{Obs}
    Throughout this document, we make a distinction between the negation (equivalence) in a quantale, denoted by $\sim$ ($\equiv$), and the negation (equivalence) in a Heyting algebra, denoted by $\lnot$ ($\leftrightarrow$). 
\end{Obs}

\begin{Teo}[\cite{Moncayo2023}, Theorem 1.2.36]\label{cuantales1.5}
    Let $\mathbb{Q}$ be a commutative integral Quantale and $x, y, y_i, x_i \in \mathbb{Q}$ with $i\in I$. Then:
    \begin{enumerate}
        \item $x\cdot \left(\bigwedge\limits_{i\in I} y_i\right)\leq \bigwedge\limits_{i\in I}(x\cdot y_i)$
        \item $x\rightarrow \left(\bigwedge\limits_{i\in I} y_i\right)= \bigwedge\limits_{i\in I}(x\rightarrow y_i)$.
        \item $\left(\bigvee\limits_{i\in I} x_i\right)\rightarrow y=\bigwedge\limits_{i\in I}(x_i\rightarrow y)$.
        \item $\sim\left(\bigvee\limits_{i\in I} x_i\right)=\bigwedge\limits_{i\in I}(\sim x_i)$.
    \end{enumerate}
\end{Teo}

\begin{Teo}[\cite{BuPi2014}, p. 2]\label{cuantales2}
    Let $\mathbb{Q}=(\mathbb{Q},\land,\lor,\cdot,\rightarrow, 1, 0)$ be a commutative integral Quantale and let $x, y, y_i, x_i \in \mathbb{Q}$ for $i\in I$. Then:
    \begin{enumerate}
        \item $x\cdot (\sim x)=0$, but in general it is not true that $x\lor \sim x=1$.
         \item $x\leq (\sim \sim x)$, but in general it is not true that $\sim \sim x\leq x$.
         \item $\sim (x\lor y)=(\sim x)\cdot (\sim y)$ (De Morgan's Law), but it is not generally true that $\sim (x\cdot y)=(\sim x)\lor (\sim y)$.
         \item If $x\leq y$, then $(\sim y)\leq (\sim x)$ and $(\sim\sim x)\leq (\sim\sim y)$.
         \item $\sim 0=1$ and $\sim 1=0$.
        \item $x=y$ if and only if $(x\equiv y)=1$.
        \item $(\sim\sim x)\cdot (\sim\sim y)\leq (\sim\sim(x\cdot y))$.
        \item $\sim\sim\sim x=\sim x$.
    \end{enumerate}
\end{Teo}

If we consider a commutative integral Quantale that is also \textbf{ idempotent}, then the structure  collapses to a Heyting algebra.
For this reason, we are interested in studying commutative integral Quantales that are not necessarily idempotent.

\section{Residuated Logic}\label{residuatedlogic}

This logic was introduced by Ulrich Höhle \cite{Hohle1994} under the name {\it Monoidal Logic} in order to present a general framework for the study of {\it Fuzzy Logics} based on t-norms, {\it Intuitionistic Logic} and {\it Girard’s Linear Logic}. In his article, Höhle considers residuated integral commutative 1-monoids (i.e. complete residuated lattices in our terms) as a set of truth values of his logic, presents a {\it completeness} and {\it soundness theorem}, and shows some interactions of it with the other logics mentioned.

Throughout this work, we will call Höhle’s {\it Monoidal Logic} as {\it Residuated Logic}, following Lano’s notation \cite{Lano1992a} in his study of Residuated Logic and Fuzzy sets, where this logic is studied in its modal variant and is applied in the context of set-theoretic models valued on Residuated lattices. \\
    
    \begin{Def}[Logical symbols, \cite{Lano1992a}]\label{StrongConjunction}
        The fundamental difference between Classical (or Intuitionistic) Logic and Fuzzy (or Residuated) Logic is that we consider additional logical symbols, namely, in Residuated Logic, we consider two types of \textbf{conjunctions}, a \textbf{weak conjunction} ($\land$) and a \textbf{strong conjunction} ($\&$). 
        
        The following symbols are definable, using the usual logical symbols:
        \begin{enumerate}
            \item Equivalence $\varphi \equiv \psi:=(\varphi \rightarrow \psi) \& (\psi \rightarrow \varphi)$,
            \item Negation $\sim \varphi := \varphi \rightarrow \bot$,
            \item Tautology $\top:=\sim \bot$.
        \end{enumerate}
    \end{Def}
    \begin{Def}[Propositional formulas, \cite{Lano1992a}]\label{residuatedformulas}
        The construction of formulas is done by recursion on a manner analogous to how it is done in Classical Propositional Logic. To differentiate these formulas from formulas in Classical (or Intuitionistic)  Logic, we call them \textit{Residuated (Propositional) formulas}, or $R$-formulas for short. 
    \end{Def}

The language of Modal Residuated Logic has the same symbols as the Residuated Logic together with an unary connective of \textit{possibility} $\Diamond$. The idea is to interpret some \textbf{quantic nucleus} as the \textbf{possibility} operator when we deal with the semantics. The formulas in a given first-order language $\mathcal{L}$ are constructed by recursion in the same way as for the residuated case and we  call them formulas \textit{Modal Residuated $\mathcal{L}$-formulas}, or \textit{$MR-\mathcal{L}$-formulas}, for short.

\subsection{Filters and nuclei on Quantales}    

Throughout this subsection, $\mathbb{Q}$ denotes a \textbf{commutative integral quantale}. The goal of this subsection is to generalize well known results for filters in Heyting algebras in the context of commutative integral quantale. That is, we use filters on Quantales to construct new Quantales (or even Heyting algebras or Boolean algebras) related to the original one in some crucial ways. These applications will be studied in subsection \ref{ResiduatedKripkemodelsofsettheory}

\begin{Def}[\cite{BuPi2014}, Definition 2]\index{filter! on a quantale}
    A nonempty subset $\mathcal{F}$ of $\mathbb{Q}$ is said to be a \textit{filter} if for every $x, y\in \mathbb{Q}$
    \begin{enumerate}
        \item If $x\leq y$ and $x\in \mathcal{F}$, then $y\in \mathcal{F}$.
        \item If $x, y\in \mathcal{F}$, then $x\cdot y\in \mathcal{F}$.
    \end{enumerate}
    A filter $\mathcal{F}$ is called \textit{proper} if $\mathcal{F}\not=\mathbb{Q}$, that is, if $0\not \in \mathcal{F}$.
\end{Def}

Consider a filter $\mathcal{F}$ on $\mathbb{Q}$. The relation $\approx_{\mathcal{F}}$ defined on $\mathbb{Q}$ by 
\begin{center}
    $x\approx_{\mathcal{F}} y$ if and only if both $x\rightarrow y$ and $y\rightarrow x\in \mathcal{F}$ 
\end{center}
is an equivalence relation on $\mathbb{Q}$.

\begin{Def}[\cite{BuPi2014}, pp. 2 and 3]\index{quotient! of a residuated lattice}\label{Q/F}
    The quotient algebra $\mathbb{Q}/_{\approx_{\mathcal{F}}}$, denoted by $\mathbb{Q}/\mathcal{F}$, becomes a complete residuated lattice in a natural way, with the operations induced from those of $\mathbb{Q}$ as follows: For $x\in \mathbb{Q}$, we denote by $|x|:=|x|_{\mathcal{F}}$ the congruence class of $x$ modulo $\approx_{\mathcal{F}}$ and we define:
    \begin{center}
        $|A|\land |B|:=|A\land B|$,\\
        $|A|\lor |B|:=|A\lor B|$,\\
        $|A|\rightarrow |B|:=|A\rightarrow B|$,\\
        $|\sim A|:=\sim|A|$.\\
    \end{center}
    The order relation on $\mathbb{Q}/\mathcal{F}$ is defined by
    \begin{center}
        $|x|\leq |y|$, if and only if, $x\rightarrow y\in \mathcal{F}$
    \end{center}
\end{Def}

But with respect to quantic nucleus, our goal is to capture, as best as we can, all the properties of the double negation $\lnot \lnot$ so that we can replicate Fitting's results \cite{Fitting1969} in the context of Residuated Logics.

\begin{Def}[\cite{Rosenthal1990}, Definition 1.1.2]\label{clousureoperatior}\index{closure operator}
    Let $(\mathbb{P}, \leq )$ be a poset. We say that a function $j:\mathbb{P}\rightarrow \mathbb{P}$ is a \textit{closure operator} if for every $x, y\in \mathbb{P}$, we have the following:
    \begin{enumerate}
        \item Expansivity: $x\leq j(x)$.
        \item Idempotency with respect to compositions: $j(j(x))=j(x)$.
        \item Monotonicity: If $x\leq y$, then $j(x)\leq j(y)$.
    \end{enumerate}
\end{Def}

\begin{Def}[\cite{Rosenthal1990}, Definition 3.1.1]\label{quanticnucleus}\index{quantic nucleus}
    We  say that a closure operator $\gamma:\mathbb{Q}\rightarrow \mathbb{Q}$ is a \textit{quantic nucleus} if for every $x\in\mathbb{Q}$
    \begin{center}
        $\gamma(x)\cdot \gamma(y)\leq \gamma(x\cdot y)$ 
    \end{center}
\end{Def}

\begin{Lema}[\cite{Rosenthal1990}, p. 29]\label{quanticnucleusproperty1}
    If $\gamma:\mathbb{Q}\rightarrow \mathbb{Q}$ is a quantic nucleus, then for every $x, y\in \mathbb{Q}$
    \begin{center}
        $\gamma(x\cdot y)=\gamma(\gamma (x)\cdot \gamma(y))=\gamma(\gamma (x)\cdot y)=\gamma(x\cdot \gamma(y))$
    \end{center}
\end{Lema}

\begin{Lema}\label{charquanticnucleus}[\cite{Rosenthal1990}, Proposition 3.1.1]
    A function $\gamma:\mathbb{Q}\rightarrow \mathbb{Q}$ is a quantic nucleus, if and only if, for every $x, y\in\mathbb{Q}$, $\gamma(x)\rightarrow\gamma(y)=x\rightarrow \gamma(y)$.
\end{Lema}

\begin{Teo}[\cite{Rosenthal1990}, Theorem 3.1.1 + Lemma 3.2.1 + Lemma 3.2.2 ]\label{fixedpointsonQ}
    Let $\gamma:\mathbb{Q}\rightarrow \mathbb{Q}$ be a quantic nucleus on $\mathbb{Q}$. The set of fixed points of $\gamma$
    \begin{center}
        $\mathbb{Q}_{\gamma}:=\{x\in\mathbb{Q}:\gamma(x)=x\}$
    \end{center}
    is a commutative integral Quantale with the order inherited from $\mathbb{Q}$ and the product 
    \begin{center}
        $x\cdot_{\gamma} y=\gamma(x\cdot y)$.        
    \end{center}
    Furthermore, if $\bigvee\limits^{\gamma}$ and $\bigwedge\limits^{\gamma}$ represent the supremums and infimums calculated in $\mathbb{Q}_{\gamma}$, respectively, then for every $x_i\in \mathbb{Q}_{\gamma}$ with $i\in I$
      \begin{center}
          $\bigvee\limits^{\gamma}_{i\in I}x_i=\gamma\left(\bigvee\limits_{i\in I}x_i\right)$\\
          $\bigwedge\limits^{\gamma}_{i\in I}x_i=\bigwedge\limits_{i\in I}x_i$
      \end{center}      
\end{Teo}

\begin{Coro}[\cite{Rosenthal1990}, Chapter 3, Section 1, Corollary 1]\label{gammaimplication}
     If $\gamma:\mathbb{Q}\rightarrow \mathbb{Q}$ is a quantic nucleus, then for every $x, y\in \mathbb{Q}$
    \begin{center}
        $\gamma(x\rightarrow y)\leq x\rightarrow \gamma(y)=\gamma(x)\rightarrow \gamma(y)$
    \end{center}
\end{Coro}

\begin{Teo}[\cite{Rosenthal1990}, Proposition 3.1.2]\label{Qgammaisclosedunder}
    If $A\subseteq \mathbb{Q}$, then $A=\mathbb{Q}_{\gamma}$ for some quantic nucleus $\gamma$ if and only if $A$ is closed under infimums and if $x\in \mathbb{Q}$ and $y\in A$, then $x\rightarrow y\in A$. 
\end{Teo}

\begin{Ej}\label{doublenegationisquanticnucleus}
    By Theorem \ref{cuantales2} items 2., 4., 7. and 8., it is clear that the operator $\sim \sim$ is a quantic nucleus on a $\mathbb{Q}$. Furthermore, by Lemma \ref{charquanticnucleus} and Theorem \ref{Qgammaisclosedunder}, we have that
    \begin{center}
        $x\rightarrow \sim\sim y=\sim\sim x\rightarrow \sim\sim y$ and $\sim\sim(\sim\sim x\rightarrow \sim\sim y)=\sim\sim x\rightarrow \sim\sim y$
    \end{center}
\end{Ej}

\begin{Def}[\cite{Moncayo2023}, Definition 1.2.50.]\index{quantic nucleus! respects implications}\label{quanticnucleusrespectsimplications}
    We  say that a quantic nucleus $\gamma:\mathbb{Q}\rightarrow \mathbb{Q}$ \textit{respects implications if}
    \begin{center}
        $\gamma(x\rightarrow y)=1$ if and only if $\gamma(x)\rightarrow \gamma(y)=1$
    \end{center}
    Notice that by Lemma \ref{charquanticnucleus}, this condition is equivalent to 
    \begin{center}
        $\gamma(x\rightarrow y)=1$ if and only if $x\rightarrow \gamma(y)=1$
    \end{center}
\end{Def}

\begin{Lema}[\cite{Moncayo2023}, Lemma 1.2.51, cf. \cite{Fitting1969} Lemma 5.3]\label{gammameetimplication}
    Let $\gamma:\mathbb{Q}\rightarrow \mathbb{Q}$ be a quantic nucleus that respects implications. For every $x_i, y\in \mathbb{Q}$, with $i\in I$,
    \begin{center}
        $\gamma(\bigwedge\limits_{i\in I}(x_i\rightarrow y))=1$ if and only if $\bigwedge\limits_{i\in I}\gamma(x_i\rightarrow y)=1$
    \end{center}
\end{Lema}
\begin{proof}

    Notice that, since $\gamma$ respects implications and by Theorem \ref{cuantales1.5} part 3.,
    \begin{center}
        $1=\gamma(\bigwedge\limits_{i\in I}(x_i\rightarrow y))=\gamma((\bigvee\limits_{i\in I}x_i)\rightarrow y)$
    \end{center}
    holds, if and only if
    \begin{center}
        $1=(\bigvee\limits_{i\in I}x_i)\rightarrow \gamma(y)$ holds.
    \end{center}
    But then, since $\gamma$ respects implications, the line given above is equivalent to
    \begin{center}
        $1=\gamma((\bigvee\limits_{i\in I}x_i)\rightarrow y)=\bigwedge\limits_{i\in I}(x_i\rightarrow \gamma(y))$.
    \end{center}
    Thus, by definition of infimum, the line given above holds, if and only if,
    \begin{center}
        $x_i\rightarrow \gamma(y)=1$ for every $i\in I$
    \end{center}
    and since $\gamma$ respects implications, that is equivalent to
    \begin{center}
        $\gamma(x_i\rightarrow y)=1$ for every $i\in I$.
    \end{center}
    which is just 
    \begin{center}
        $\bigwedge\limits_{i\in I}\gamma(x_i\rightarrow y)=1$
    \end{center}
\end{proof}

\begin{Def}[\cite{Rosenthal1990}, Definition 3.2.4.]\index{quantic nucleus! idempotent with respect to products }
    We  say that a quantic nucleus $\gamma:\mathbb{Q}\rightarrow \mathbb{Q}$ is \textit{idempotent with respect to products} if for every $x\in\mathbb{Q}$
    \begin{center}
        $\gamma(x^2)=\gamma(x)$
    \end{center}
\end{Def}

\begin{Teo}[\cite{Rosenthal1990}, Lemma 3.2.3]\label{standardisheyting}
    Let $\gamma:\mathbb{Q}\rightarrow \mathbb{Q}$ be a quantic nucleus idempotent with respect to products on $\mathbb{Q}$. Then, $\mathbb{Q}_{\gamma}=\{x\in\mathbb{Q}:\gamma(x)=x\}$ is a idempotent commutative integral quantale, that is, $\mathbb{Q}_{\gamma}$ is a Heyting algebra.
\end{Teo}    

\begin{Def}[\cite{Moncayo2023}, Definition 1.2.55.]\index{quantic nucleus! respects the bottom element }\label{quanticnucleusrespectsbottomelement}
    We  say that a quantic nucleus $\gamma:\mathbb{Q}\rightarrow \mathbb{Q}$  \textit{respects the bottom element} if 
    \begin{center}
        $\gamma(0)=0$.
    \end{center}
\end{Def}

\begin{Lema}[\cite{Moncayo2023}, Lemma 1.2.56]
\label{doublenegationgammaisidempotent}
    Let $\gamma:\mathbb{Q}\rightarrow \mathbb{Q}$ be a quantic nucleus that respects the bottom element. Then,
    \begin{center}
        $\sim \sim \gamma(x)=\sim \sim (\gamma ( \sim \sim \gamma(x))$,
    \end{center}
    that is, $\sim \sim\gamma$ is idempotent.
\end{Lema}
\begin{proof}
    Notice that $\gamma(0)=0\in\mathbb{Q}_{\gamma}$ and by Theorem \ref{Qgammaisclosedunder}, we can deduce that for every $z\in \mathbb{Q}$, $z\rightarrow 0=\sim z\in \mathbb{Q}_{\gamma}$ and thus $\sim\sim z\in \mathbb{Q}_{\gamma}$, that is, 
    \begin{center}
        $\gamma(\sim \sim z)=\sim \sim z$.
    \end{center}
    Therefore, by taking $z:=\gamma(x)$ and by Theorem~\ref{cuantales2} item 8.  we have that
    
    \begin{center}
        $\sim \sim (\gamma ( \sim \sim \gamma(x))=\sim \sim (\sim \sim \gamma(x))= \sim \sim \gamma(x)$.
    \end{center}
\end{proof}

\begin{Def}[\cite{Moncayo2023}, 1.2.57.]\index{quantic nucleus! standard}
    We say that a quantic nucleus $\gamma:\mathbb{Q}\rightarrow \mathbb{Q}$ is \textit{standard} if $\gamma$ is idempotent with respect to products, respects implications and the bottom element.
\end{Def}

We introduce the new notion of \textbf{standard quantic nucleus} since it captures all the necessary properties from the double negation operator that are
used in Section \ref{subsectionresiduatedkripkemodels} to generalize the results given in \cite{Fitting1969} that we need. 

\begin{Teo}[\cite{Moncayo2023}, Theorem 1.2.58.]\label{condensedisfilter}
    Let $\gamma:\mathbb{Q}\rightarrow \mathbb{Q}$ be a quantic nucleus on $\mathbb{Q}$. The set
    \begin{center}
        $\mathcal{F}_{\gamma}:=\{x\in \mathbb{Q}: \gamma(x)=1\}$
    \end{center}
    is a filter on $\mathbb{Q}$.
\end{Teo}
\begin{proof}
\begin{enumerate}    
    \item Since $\gamma$ is expansive, $1\leq \gamma(1)$, and then $1\in \mathcal{F}_{\gamma}$
    \item Take $x\in \mathcal{F}_{\gamma}$ and $y\in \mathbb{Q}$ such that $x\leq y$. Since $\gamma$ is monotone function, $\gamma(x)\leq \gamma(y)$ and since $x\in \mathcal{F}_{\gamma}$, $\gamma(x)=1$. Then, $\gamma(y)=1$ and $y\in \mathcal{F}_{\gamma}$.
    \item Take $x, y\in \mathcal{F}_{\gamma}$, that is, $\gamma(x)=\gamma(y)=1$. Then, since $\gamma$ is a quantic nucleus, $\gamma(x\cdot y)\geq \gamma(x)\cdot \gamma(y)=1\cdot 1=1$, that is, $x\cdot y\in \mathcal{F}_{\gamma}$.    
\end{enumerate}
\end{proof}

For the rest of the section, $|x|$ denotes the class of $x\in \mathbb{Q}$ modulo $\approx_{\mathcal{F}_{\gamma}}$.

\begin{Teo}[\cite{Moncayo2023}, Theorem 1.2.59.]\label{ABiffjAJB}
    Let $\gamma$ be a quantic nucleus that respects implications and take $x, y\in \mathbb{Q}$. Then, $|x|=|y|$ if and only if $\gamma(x)=\gamma(y)$.
\end{Teo}
\begin{proof}
    \begin{align*}
            |x|=|y|&\text{ iff } x\rightarrow y\in \mathcal{F}_{\gamma} \mbox{ and } y\rightarrow x\in \mathcal{F}_{\gamma} &&\text{(by definition of $\approx_{\mathcal{F}_{\gamma}}$)} \\
            &\text{ iff } \gamma(x\rightarrow y)=1=\gamma(y\rightarrow x)\hspace{0.1cm} &&\mbox{(by definition of $\mathcal{F}_{\gamma}$)}\\
            & \text{ iff } \gamma(x)\rightarrow \gamma(y)=1=\gamma(y)\rightarrow \gamma(x)\hspace{0.1cm} &&\mbox{($\gamma$ respects implications)} \\
            &\text{ iff }  \gamma(x)\leq \gamma(y) \text{ and } \gamma(y)\leq \gamma(x) \hspace{0.1cm} &&\mbox{(by  Theorem \ref{cuantales1} item 1.)}\\
            & \text{ iff } \gamma(x)=\gamma(y)
        \end{align*}
\end{proof}

\begin{Coro}[\cite{Moncayo2023}, Corollary 1.2.60.]\label{clasedegammaxesclasedex}
    If $\gamma$ is a quantic nucleus that respects implications and $x\in \mathbb{Q}$, 
    \begin{center}
        $|\gamma(x)|=|x|$,
    \end{center}
    since $\gamma(\gamma(x))=\gamma(x)$.
\end{Coro}

    \begin{Teo}[\cite{Moncayo2023}, Theorem 1.2.61. ]\label{FgammaisHeyting}
        Let $\gamma$ be a quantic nucleus idempotent with respects to products and that respects implications. Then, $\mathbb{Q}/\mathcal{F}_{\gamma}$ is a Heyting algebra.     
    \end{Teo}

    \begin{proof}
        By Theorems \ref{Q/F} and \ref{condensedisfilter}, we know that $\mathbb{Q}/\mathcal{F}_{\gamma}$ is a complete residuated lattice. 
        
        It is enough to prove that the product in $\mathbb{Q}/\mathcal{F}_{\gamma}$ is idempotent. If $|x|\in \mathbb{Q}/\mathcal{F}_{\gamma}$, then
            \begin{align*}
                |x|\cdot |x|&= |x\cdot x|\hspace{0.1cm}\small &&\mbox{(by definition of $\cdot$ on } \mathbb{Q}/\mathcal{F}_{\gamma}) \\
                &= |\gamma(x\cdot x)|\hspace{0.1cm}\small &&\mbox{(by Corollary }\ref{clasedegammaxesclasedex}) \\
                &= |\gamma(x)|\hspace{0.1cm}\small &&\mbox{(since $\gamma$ is idempotent with respect to products} )\\
                &= |x|\hspace{0.1cm}\small &&\mbox{(by Corollary }\ref{clasedegammaxesclasedex} )
            \end{align*}         
    \end{proof}

    \begin{Teo}[\cite{Moncayo2023}, Theorem 1.2.62, cf. \cite{Fitting1969}, Theorem 5.4]\label{unionclass}
        Let $\gamma$ be a quantic nucleus on $\mathbb{Q}$ that respects implications. For every $x_i\in \mathbb{Q}$ with $i\in I$, we have that
        \begin{center}
            $|\bigvee\limits_{i\in I}x_i|=\bigvee\limits_{i\in I}|x_i|$
        \end{center}
    \end{Teo}
    \begin{proof}
        For every $i\in I$, $x_i\leq \bigvee\limits_{i\in I}x_i$. Thus, 
        \begin{center}
             $x_i\rightarrow \bigvee\limits_{i\in I}x_i=1$,
        \end{center}
        therefore,
        \begin{center}
            $\gamma(x_i\rightarrow \bigvee\limits_{i\in I}x_i)=1$,
        \end{center}
        and by definition of $\leq$ we deduce that $|x_i|\leq |\bigvee\limits_{i\in I}x_i|$. Thus, $|\bigvee\limits_{i\in I}x_i|$ is an upper bound of $\{|x_i|:i\in I\}$.
        
        To see that it is the smallest upper bound, take $|b|\in \mathbb{Q}/\mathcal{F}_{\gamma}$ an upper bound of $\{|x_i|:i\in I\}$, that is,
        
        \begin{center}
            $|x_i|\leq |b|$ for every $i\in I$,
        \end{center}
        that implies, by definition of $\leq$ on $\mathbb{Q}/\mathcal{F}_{\gamma}$, 
        \begin{center}
            $\gamma(x_i\rightarrow b)=1$ for every $i\in I$.
        \end{center}
        Therefore, 
        \begin{center}
            $\bigwedge\limits_{i\in I}(\gamma(x_i\rightarrow b))=1$,
        \end{center}
        then, by Lemma \ref{gammameetimplication}
        \begin{center}
             $\gamma(\bigwedge\limits_{i\in I}(x_i\rightarrow b))=1$,
        \end{center}
        hence, by Theorem \ref{cuantales1.5} item 3,
        \begin{center}
             $\gamma((\bigvee\limits_{i\in I}x_i)\rightarrow b)=1$,
        \end{center}        
        thus, by definition of $\leq$,        
        \begin{center}
            $|\bigvee\limits_{i\in I}x_i|\leq  |b|$
        \end{center}
    \end{proof}

\section{Residuated Kripke models}\label{subsectionresiduatedkripkemodels}

In \cite{OnoKomori}, Ono and Komori generalize the notion of Intuitionistic Kripke models for (propositional) logics without contraction, that is, substructural logics without the idempotency of the conjunction. Then, in \cite{Ono1985}, Ono defines the notion of Kripke models for the predicate case. For the most part, we follow their notation and conventions, but with small variations, since we are not interested in working with Gentzen-style sequences. Even though these models were initially used for the study of substructural logics, MacCaull \cite{MacCaull1996} showed how the models in \cite{OnoKomori} can be seen as models for Residuated Logic.

Since we are working in the context of Residuated Logic, we need a more robust structure than just a poset $(\mathbb{P}, \leq)$ in order to properly interpret the operation of strong conjunction $\&$ and to capture the subtleties of the Residuated Logic, so it would be natural to work with some kind of \textbf{ordered monoid}, just as Quantales are defined. But there is a small issue: in the Kripke semantics convention, we would expect the forcing relation to preserve truth \textbf{upwards}, that is:

\begin{center}
    if a sentence $\varphi$ is forced at $p$ and $p\leq q$, then the sentence is forced at $q$, where $p, q\in \mathbb{P}$.
\end{center}

So if we were to consider the order as we have been doing it with Quantales, this property would hold \textbf{backwards} (i.e. there would be truth preservation \textbf{downwards}). That is why some of the properties that we require for our orders in this section are the dual ones of the properties of Quantales.

\begin{Def}[\cite{OnoKomori}, Section 3]
    We say that $(\mathbb{P}, \leq, \cdot, 1)$ is a \textit{partially ordered commutative monoid} if
    \begin{enumerate}
        \item $(\mathbb{P}, \leq, 1)$ is a partial order with $1$ as the bottom element.
        \item $(\mathbb{P}, \cdot, 1)$ is a commutative monoid.
        \item For all $a, b, c \in \mathbb{P}$, if $a\leq b$, then $a\cdot c\leq b\cdot c$. 
    \end{enumerate}
\end{Def}

\begin{Def}[\cite{OnoKomori}, Section 3]
    We say that $(\mathbb{P}, \leq, \land , \cdot, 1)$ is an \textit{SO-commutative monoid} if
    \begin{enumerate}
        \item $(\mathbb{P},\land, \leq)$ is a meet-semilattice.
        \item $(\mathbb{P}, \leq, \cdot, 1)$ is a partially ordered commutative monoid.
    \end{enumerate}
\end{Def}

\begin{Def}[\cite{Ono1985}, p. 189]\label{completeSO-commutativemonoid}
    We say that an SO-commutative monoid $(\mathbb{P}, \leq, \land , \cdot, 1)$ is \textit{complete} if
    \begin{enumerate}
        \item $(\mathbb{P},\land)$ is a complete meet-semilattice (thus, $\mathbb{P}$ has a top element, denoted by $\infty$)
        \item For every $a, b_i\in M$ with $i\in I$, $a\cdot \bigwedge\limits_{i\in I} b_i= \bigwedge\limits_{i\in I} (a\cdot b_i)$
    \end{enumerate}
\end{Def}

\begin{Obs}\label{implicationonP}
    Notice that since $a\_\cdot$ preserves arbitrary meets, by the Adjoint Functor Theorem for preorders, $a\cdot\_$ is a right adjoint, that is, there exists a function $a\rightarrow\_:\mathbb{P}\rightarrow \mathbb{P}$ such that for all $b, c\in \mathbb{P}$
    \begin{center}
        $a\rightarrow c\leq b$, if and only if, $c\leq a\cdot b$
    \end{center}
    Furthermore, this implication allow us to define a negation on $\mathbb{P}$ by taking
    \begin{center}
        $\sim a:= a\rightarrow \infty$
    \end{center}
    Notice that we take $a\rightarrow \infty$ as the definition of $\sim a$ since $\infty$ is the top element of $\mathbb{P}$, and the definitions and properties on this order are the duals of those in the context of Quantales. 
\end{Obs}

\begin{Teo}[\cite{Moncayo2023}]\label{propertiesSOmonoidP}
    Take an SO-commutative complete monoid $\mathbb{P}=(\mathbb{P}, \leq, \land , \cdot, 1, \infty)$. Then, $\mathbb{P}$ satisfies (for the most part) the dual properties of $\mathbb{Q}$. More specifically, if $a, b, c \in \mathbb{P}$ with $i\in I$, then:
    \begin{enumerate}
        \item $a\leq b$, if and only if, $(b\rightarrow a)=1$.
        \item $b\leq a\cdot(a\rightarrow b)$.
        \item If $a\leq b$, then $b\rightarrow c\leq a\rightarrow c$.
        \item If $a\leq b$, then $c\rightarrow a\leq c\rightarrow b$.
        \item $(a\cdot b)\rightarrow c= a\rightarrow(b\rightarrow c)$.
        \item $a\cdot (\sim a)=\infty$.
    \end{enumerate}
\end{Teo}

We focus on Residuated Kripke models with constant universe (i.e. $\mathcal{D}(p)=D$ for all $p\in \mathbb{P}$). We follow the presentation of \cite{Ono1985}, where he calls this kind of models \textit{total strong frame with constant domain} or simply \textit{total CD-frame}. Since all the Residuated Kripke models that we consider have constant universe, we will not call attention on this fact from now on.

\begin{Def}[Residuated Kripke Model, \cite{Ono1985}, p. 189]\index{residuated! Kripke model}\label{ResiduatedKripkeModel}
    We say that $\mathcal{A}=(\mathbb{P}, \leq, \Vdash, D)$ is a \textit{Residuated Kripke $\mathcal{L}$-model} (or \textit{$R$-Kripke $\mathcal{L}$-model} for short) if
    \begin{enumerate}
        \item $\mathbb{P}=(\mathbb{P}, \leq, \land , \cdot, 1, \infty)$ is a complete SO-commutative monoid.
        \item $\Vdash$ is a relation between elements of $\mathbb{P}$ and atomic sentences in the language
        \begin{center}
            $\mathcal{L}_{\mathcal{A}}=\mathcal{L}\cup D$,
        \end{center}
         where each element of $D$ is considered as a constant symbol. We denote $(p, \varphi) \in \ \Vdash$ by  $\mathcal{A}\Vdash_p \varphi$ and we say that \textit{$\varphi$ is forced in $\mathcal{A}$ at $p$}.
        \item  Given $p_i, q\in \mathbb{P}$, with $i\in I$ and $\varphi$ an atomic $\mathcal{L}_{\mathcal{A}}$-sentence, we require that $\Vdash$ satisfies the following conditions:
        \begin{enumerate}[label=\alph*.]
            \item If $\bigwedge\limits_{i\in I}{p_i} \leq q$ and for each $i\in I$ $\mathcal{A}\Vdash_{p_i} \varphi$, then $\mathcal{A}\Vdash_q \varphi$.
            \item $\mathcal{A}\Vdash_{\infty}\varphi$ for every atomic $R-\mathcal{L}_{\mathcal{A}}$-sentence $\varphi$.
            \item $\mathcal{A}\Vdash_{p} \bot$ if and only if $p=\infty$. Recall that $\bot$ is the symbol of contradiction.
        \end{enumerate}
    \end{enumerate}
\end{Def}

\begin{Def}[Residuated Kripke forcing \cite{Ono1985}, p. 189]\index{residuated! Kripke forcing}\label{ResiduatedKripkeForcing}
    Given  $\mathcal{A}=(\mathbb{P}, \leq, \Vdash, D)$ a Residuated Kripke $\mathcal{L}$-model, $p\in \mathbb{P}$ and an $R-\mathcal{L}_{\mathcal{A}}$-sentence $\varphi$, we can extend the forcing relation $\mathcal{A}\Vdash_p \varphi$ to all $R-\mathcal{L}_{\mathcal{A}}$-sentences by recursion on the complexity of $\varphi$

    \begin{enumerate}
        \item $\mathcal{A}\Vdash_p (\varphi \&\psi)$, if and only if, there are $q, r\in\mathbb{P}$ such that $p\geq q\cdot r$, $\mathcal{A}\Vdash_q \varphi$ and $\mathcal{A}\Vdash_r \psi$.
        
        \item $\mathcal{A}\Vdash_p (\varphi \lor \psi)$, if and only if, there are $q, r\in\mathbb{P}$ such that $p\geq q\land r$, and both ($\mathcal{A}\Vdash_q \varphi$ or $\mathcal{A}\Vdash_q \psi$) and ($\mathcal{A}\Vdash_r \varphi$ or $\mathcal{A}\Vdash_r \psi$) hold.
        
        \item $\mathcal{A}\Vdash_p (\varphi \land \psi)$, if and only if, $\mathcal{A}\Vdash_p \varphi$ and $\mathcal{A}\Vdash_p \psi$.
        
        \item $\mathcal{A}\Vdash_p (\varphi\rightarrow \psi)$, if and only if, for all $q, r\in\mathbb{P}$ if $\mathcal{A}\Vdash_q \varphi$ and $p\cdot q\leq r$, then $\mathcal{A}\Vdash_{r} \psi$.
        
        \item $\mathcal{A}\Vdash_p \exists x\varphi(x)$ if and only there exists an index set $I$ such that for every $i\in I$, and there exist $d_i\in D$ and $q_i\in \mathbb{P}$ such that $\bigwedge\limits_{i\in I}q_i\leq p$ and $\mathcal{A}\Vdash_{q_i}\varphi(d_i)$
        
        \item $\mathcal{A}\Vdash_p \forall x\varphi(x)$, if and only if,  for all $b\in D$,   $\mathcal{A}\Vdash_p \varphi(b)$.
    \end{enumerate}
\end{Def}
\begin{Obs}
    By definition of $\sim \varphi:=\varphi\rightarrow \bot$, it is straightforward to prove that
    \begin{center}
        $\mathcal{A}\Vdash_p \sim\varphi$, if and only if, for all $q, r\in\mathbb{P}$ if $\mathcal{A}\Vdash_q \varphi$ and $p\cdot q\leq r$, then $r=\infty$.
    \end{center}
\end{Obs}
\begin{Obs}
    Ono and Komori (see \cite{OnoKomori} Section 6) proved that the propositional version of these models generalizes the usual Intuitionistic Propositional Kripke Models.
\end{Obs}

\begin{Def}[\cite{Ono1985}, p. 189]
    Given $\mathcal{A}=(\mathbb{P}, \leq, \Vdash, D)$ an $R$-Kripke $\mathcal{L}$-model and an $R-\mathcal{L}_{\mathcal{A}}$-sentence $\varphi$, we say that $\varphi$ is \textit{true} in the structure $\mathcal{A}$ (denoted by $\mathcal{A}\models \varphi$) if for all $p\in \mathbb{P}$, $\mathcal{A}\Vdash_p \varphi$. 
\end{Def}

\begin{Teo}[Completeness and Soundness Theorem for the Kripke semantics, \cite{Ono1985} Lemma 2.2 and Theorem 2.3]
     Given an $R-\mathcal{L}$-theory $T$ and $\varphi$ an $R-\mathcal{L}$-sentence, we have that $T\vdash_r \varphi$, if and only if, for every $R$-Kripke $\mathcal{L}$-model $\mathcal{A}$, if $\mathcal{A}\models T$, then $\mathcal{A}\models \varphi$.
\end{Teo}

Now we proceed to analyze the behavior of the sets of conditions that force a given formula in the context of these models.

\begin{Def}[\cite{OnoKomori}, p. 194]\label{capclosed}\index{$\cap$-closed}
    We say that $A\subseteq \mathbb{P}$ is \textit{$\cap$-closed} $A$ if $A$ is hereditary and closed under (finite) meets, that is:
    \begin{enumerate}
        \item For all $a\in A$ and $b\in \mathbb{P}$, if $a\leq b$, then $b\in A$.
        \item For all $a, b\in A$, $a\land b\in A$.
    \end{enumerate}
    We denote
    \begin{center}
         $D(\mathbb{P}):=\{A\subseteq \mathbb{P}: A \text{ is }\cap \text{-closed}\}$
    \end{center}
\end{Def}

Ono and Komori work with the notion of $\cap$-closed in \cite{OnoKomori} for the propositional case. Since we work in the predicate version of this kind of logics, we have to change the notion of $\cap$-closed to a new notion that we called \textbf{strongly hereditary} set. This had to be done, since in the case of the Predicate Logic the set of elements of $\mathbb{P}$ that force a formula are characterized by a stronger assumption.

\begin{Def}\label{stronglyhereditary}
    We say that a \textbf{non-empty} set $A\subseteq \mathbb{P}$ is \textit{strongly hereditary} if for all $c_i\in A$ and $d\in \mathbb{P}$ for $i\in I$, if $\bigwedge\limits_{i\in I} c_i\leq d$, then $d\in A$. Notice that from the definition it is clear that if $A$ is strongly hereditary, then $A$ is hereditary. We denote
    \begin{center}
         $\mathbb{P}^*:=\{A\subseteq \mathbb{P}: A \text{ is strongly hereditary}\}$
    \end{center}
\end{Def}

As we mentioned before, we need to work with the notion of strongly hereditary and not with Ono's notion of $\cap$-closed since the set

\begin{center}
    $\{p\in \mathbb{P}: \mathcal{A}\Vdash_p \varphi\}$
\end{center}

is strongly hereditary and not just $\cap$-closed for a given $R-\mathcal{L}$-sentence $\varphi$.

\begin{Def}[\cite{Moncayo2023}, Definition 3.1.25, cf. \cite{OnoKomori}, p. 194]\label{operationsstronglyhereditaty}
    Let $(\mathbb{P}, \leq, \land , \cdot, 1, \infty)$ be a complete SO-commutative monoid. Take $A, B, A_i\subseteq \mathbb{P}$, with $i\in I$ and $x\in \mathbb{P}$. We define
    
     \begin{center}
     $A\cdot B:=\{c\in \mathbb{P}: \text{ there exist }a\in A, b\in B \text{ such that } c\geq a\cdot b \}$\\
     
     $x\cdot A:=\{x\}\cdot A=\{c\in \mathbb{P}: \text{ there exists }a\in A \text{ such that }c\geq a\cdot x \}$\\
     
     $A\rightarrow B:=\{c\in \mathbb{P}: c\cdot A\subseteq B\}$\\

     $0_{\mathbb{P}^*}:=\{\infty\}$.

     $1_{\mathbb{P}^*}:=\mathbb{P}$.
         
     $A\lor B:=\{c\in \mathbb{P}: \text{ there exist }a, b\in A\cup B \text{ such that } c\geq a\land b \}$\\
     
     $A\land B:=A\cap B$.
     
     $\bigvee\limits_{i\in I} A_i:=\{c\in \mathbb{P}: \text{ there exists an index set }J\text{ such that for all }  j\in J, \text{ there are } a_j\in \bigcup\limits_{i\in I} A_i \text{ such that } \bigwedge\limits_{j\in J} a_j\leq c \}$\\
     
     $\bigwedge\limits_{i\in I} A_i:=\bigcap\limits_{i\in I} A_i$
     \end{center}
\end{Def}

Our definition of the arbitrary join differs from the one given in \cite{OnoKomori}. Ono and Komori defined the arbitrary join operation as follows:

\begin{center}
    $\bigvee\limits_{i\in I} A_i:=\{c\in \mathbb{P}: \text{ there exists a \textbf{finite} index set }J\text{ such that for all }  j\in J, \text{ there are } a_j\in \bigcup\limits_{i\in I} A_i \text{ such that } \bigwedge\limits_{j\in J} a_j\leq c \}.$\\
\end{center}

We had to change this, since we are working on the set $\mathbb{P}^*$ and not on the set $D(\mathbb{P})$, which happens to be larger, and therefore the definition of joins might not coincide.\\

On the other hand, Ono and Komori proved that $D(\mathbb{P})$ is a complete full BCK-algebra (see \cite{OnoKomori} for a definition of this algebra) with the operations they defined. Since a complete full BCK-algebra happens to be a complete Residuated Lattice, we use their ideas to prove that the set $\mathbb{P}^*$ is a complete Residuated Lattice with the operations that we defined above.

\begin{Teo}[\cite{Moncayo2023}, Theorem 3.1.26, cf. \cite{OnoKomori}, Lemma 8.3]\label{P*isaresiduatedlattice}
    Let $(\mathbb{P}, \leq, \land , \cdot, 1, \infty)$ be a complete SO-commutative monoid. Then, $\mathbb{P}^*$ endowed with the operations of Definition \ref{operationsstronglyhereditaty} and the order $\subseteq$ forms a complete Residuated Lattice.
\end{Teo}

\begin{Lema}[\cite{Ono1985}, Lemma 2.1]\label{sentencesstronglyhereditary}
    For every $R-\mathcal{L}_{\mathcal{A}}$-sentence $\varphi$ and $a_i, b\in \mathbb{P}$, with $i\in I$. If $\mathcal{A}\Vdash_{a_i}\varphi$ for every $i\in I$ and $\bigwedge\limits_{i\in I}a_i\leq b$, then $\mathcal{A}\Vdash_{b}\varphi$. Notice that this lemma implies that the set  $\{p\in\mathbb{P}:\mathcal{A}\Vdash_p \varphi\}$ is strongly hereditary.
\end{Lema}

\begin{Teo}[\cite{Moncayo2023}, Theorem 3.1.28, cf. \cite{Fitting1969}]\label{kripkecongruence}
    Take a $R$-Kripke $\mathcal{L}$-model $\mathcal{A}=(\mathbb{P}, \leq, \Vdash, D)$. If $\varphi$ and $\psi$ are $R-\mathcal{L}_{\mathcal{A}}$-sentences and $\theta(x)$ is an $R-\mathcal{L}_{\mathcal{A}}$-formula, then
    \begin{center}
        $\{p\in \mathbb{P}:\mathcal{A}\Vdash_p \varphi\}\cdot \{p\in \mathbb{P}:\mathcal{A}\Vdash_p \psi\}=\{p\in \mathbb{P}:\mathcal{A}\Vdash_p \varphi\& \psi\}$\\

        $\{p\in \mathbb{P}:\mathcal{A}\Vdash_p \varphi\}\lor \{p\in \mathbb{P}:\mathcal{A}\Vdash_p \psi\}=\{p\in \mathbb{P}:\mathcal{A}\Vdash_p \varphi\lor \psi\}$\\
        
        $\{p\in \mathbb{P}:\mathcal{A}\Vdash_p \varphi\}\cap \{p\in \mathbb{P}:\mathcal{A}\Vdash_p \psi\}=\{p\in \mathbb{P}:\mathcal{A}\Vdash_p \varphi\land \psi\}$

        $\{p\in \mathbb{P}:\mathcal{A}\Vdash_p \varphi\}\rightarrow \{p\in \mathbb{P}:\mathcal{A}\Vdash_p \psi\}=\{p\in \mathbb{P}:\mathcal{A}\Vdash_p \varphi\rightarrow \psi\}$
        
       $\sim\{p\in \mathbb{P}:\mathcal{A}\Vdash_p \varphi\}=\{p\in \mathbb{P}:\mathcal{A}\Vdash_p \sim\varphi\}$
        
        $\bigcap\limits_{d\in D}\{p\in \mathbb{P}:\mathcal{A}\Vdash_p \theta(d)\}=\{p\in \mathbb{P}:\mathcal{A}\Vdash_p \forall x\theta(x)\}$
        
        $\bigvee\limits_{d\in D}\{p\in \mathbb{P}:\mathcal{A}\Vdash_p \theta(d)\}=\{p\in \mathbb{P}:\mathcal{A}\Vdash_p \exists x\theta(x)\}$
    \end{center}
\end{Teo}

\section{Residuated Kripke models with modal operators}\label{ResiduatedKripkemodelswithmodaloperators}

Throughout this subsection, $\mathbb{Q}$ denotes a \textbf{complete Residuated Lattice}.

This section is inspired in the work of Lano \cite{Lano1992a}, where he considers a special kind of \textbf{interior operator} $I$ on a residuated lattice (structure that he calls \textbf{Topological Residuated Lattice}) and constructs models of \textit{Residuated Modal Set Theory}. He introduces an operator of \textbf{necessity} (denoted $\Box$) to be interpreted by an interior operator $I$ and then provides an axiomatization of this logic and proves some results for the model of Set Theory that he constructs.

Unlike Lano, we consider a special kind of \textbf{closure operator} (a \textbf{standard quantic nucleus}) on our residuated lattice and develop a similar kind of model where we augment the logic with a \textbf{possibility} operator (denoted $\Diamond$). This is done since the order that Lano considers for his models preserves truth downwards (following Cohen's convention) and not upwards (following Kripke's).

Thus, we would like to extend the definition of Kripke model so that we can interpret formulas of the form $\Diamond \varphi$. To get an idea of how we should define them, we start from the more natural definition of $\Diamond \varphi$ in a valued model. So consider a $\mathbb{Q}$-valued modal model $\mathcal{M}$ and a quantic nucleus $\gamma$ on $\mathbb{Q}$ (see Definition \ref{quanticnucleus}).

Then, by definition, we have that

\begin{center}
$\ldbrack\Diamond \varphi \rdbrack:= \gamma(\ldbrack\varphi \rdbrack)$.
\end{center}

Let us assume that the formula $\varphi$ satisfies the relation

\begin{center}
    $\mathcal{A}\Vdash_p \varphi$ iff $p\leq \ldbrack\varphi \rdbrack$ (induction hypothesis)
\end{center}

and let us try to define $\mathcal{A}\Vdash_p \Diamond\varphi$ using only the forcing relation $\Vdash$ and the quantic nucleus on $\gamma$. Therefore, we want to prove the following equivalences:

\begin{align*}
            \mathcal{A}\Vdash_p \Diamond\varphi &\text{ iff } p\leq \ldbrack\Diamond\varphi \rdbrack  \\
            &\text{ iff } p\leq \gamma(\ldbrack \varphi \rdbrack) &&\mbox{(by definition of }\ldbrack\Diamond\varphi \rdbrack)\\
            & \text{ iff } \text{there exists $q\in \mathbb{P}$ such that $p\leq \gamma(q)$ and $q\leq \ldbrack\varphi \rdbrack$} &&\mbox{(since $\gamma$ is monotone)} \\
            & \text{ iff } \text{there exists $q\in \mathbb{P}$ such that $p\leq \gamma(q)$ and $\mathcal{A}\Vdash_q \varphi$} &&\mbox{(by induction hypothesis)}
\end{align*}

Therefore, we can define $\mathcal{A}\Vdash_p \Diamond\varphi$ only in terms of forcing and some nucleus $\gamma$ on $\mathbb{P}$. But since the order of the underlining order of the Kripke model is reversed, we need to reverse the order in last line of equivalences above. Also, we have to reverse the order in the conditions defining a quantic nucleus. That lead us to the following definition.

\begin{Def}[\cite{Moncayo2023}, Definition 3.1.29.]\label{conucleusonorder}\index{conucleus}
    We say that a function $\delta:\mathbb{P}\rightarrow\mathbb{P}$ is a \textit{conucleus} on a complete SO-monoid $(\mathbb{P}, \leq, \cdot)$ if for all $p, q, p_i\in \mathbb{P}$ with $i\in I$
    \begin{enumerate}
        \item $\delta(p)\leq p$.
        \item If $p\leq q$, then $\delta(p)\leq \delta(q)$.
        \item $\delta(\delta(p))=\delta(p)$.
        \item $\delta(p\cdot q)\leq \delta(p)\cdot \delta(q)$.
        \item $\delta(\bigwedge\limits_{i\in I} p_i)=\bigwedge\limits_{i\in I} \delta(p_i)$
    \end{enumerate}
\end{Def}

\begin{Obs}
    Notice that the conditions $1-4$ are the dual ones of the conditions in the definition of a quantic nucleus (see Definition \ref{quanticnucleus}) and therefore $\delta$ is a \textbf{interior operator} rather than a closure one. The condition $5$ was added to been able to prove that:
    \begin{enumerate}
        \item The set $\{p\in \mathbb{P}: \mathcal{A}\Vdash_p \varphi\}$ is strongly hereditary for every $MR$-formula $\varphi$,
        \item The quantic nucleus $\gamma_{\delta}$ (we will define this operation in Theorem \ref{gamma*}) is well defined on the set $\mathbb{P}^*$.
    \end{enumerate}    
\end{Obs}

\begin{Obs}
    Notice that since $\delta$ distributes over arbitrary meets, by the Adjoint Functor Theorem for preorders, there exists an operator $\rho:\mathbb{P}\rightarrow \mathbb{P}$ that is the left adjoint of $\delta$. As it is usual in Modal Logic (for example in $S_5$), the necessity and the possibility form an adjoint pair. Therefore, we may use the operator $\rho$ to define a notion of \textbf{necessity} in our Kripke models. We will not do so in our work, since there is no real need for this operator for our results. 
\end{Obs}

\begin{Obs}
    This notion of conucleus is different the notion defined by Rosenthal in \cite{Rosenthal1990}. A conucleus in \cite{Rosenthal1990} is an interior operator that satisfies the condition
    \begin{center}
        $\delta(p)\cdot \delta(q)\leq \delta(p\cdot q)$.
    \end{center}
    which is the opposite of what we require. 
\end{Obs}

\begin{Def}[\cite{Moncayo2023}, Definition 3.1.33.]
    We say that $\mathbb{P}=(\mathbb{P}, \leq, \land , \cdot, 1, \infty, \delta)$ is a \textit{complete modal SO-commutative monoid} if
    \begin{enumerate}
        \item $(\mathbb{P}, \leq, \land , \cdot, 1, \infty)$ is a complete SO-commutative monoid.
        \item $\delta$ is a conucleus on $(\mathbb{P}, \leq, \cdot)$.
    \end{enumerate}
\end{Def}

\begin{Def}[Modal Residuated Kripke model, \cite{Moncayo2023} Definition 3.1.34 cf. \cite{Ono1985} p. 189]\index{modal! residuated Kripke model}\label{ModalResiduatedKripkeModel}
    We say that $\mathcal{A}=(\mathbb{P}, \leq, \delta, \Vdash, D)$ is a \textit{Modal Residuated Kripke $\mathcal{L}$-model} (or $MR$-Kripke $\mathcal{L}$-model, for short) if
    \begin{enumerate}
        \item $\mathbb{P}=(\mathbb{P}, \leq, \land , \cdot, 1, \infty, \delta)$ is a complete \textbf{modal} SO-commutative monoid.
        \item $\Vdash$ is a relation between elements of $\mathbb{P}$ and atomic sentences in the language
        \begin{center}
            $\mathcal{L}_{\mathcal{A}}=\mathcal{L}\cup D$,
        \end{center}
         where each element of $D$ is considered as a constant symbol. We denote $(p, \varphi) \in \ \Vdash$ as $\mathcal{A}\Vdash_p \varphi$.
        \item  Given $p_i, q\in \mathbb{P}$, with $i\in I$ and $\varphi$ an atomic $\mathcal{L}_{\mathcal{A}}$-sentence, we require that $\Vdash$ satisfies the following conditions:
        \begin{enumerate}[label=\alph*.]
            \item If $\bigwedge\limits_{i\in I}{p_i} \leq q$ and for each $i\in I$ $\mathcal{A}\Vdash_{p_i} \varphi$, then $\mathcal{A}\Vdash_q \varphi$.
            \item $\mathcal{A}\Vdash_{\infty}\varphi$ for every atomic $\mathcal{L}_{\mathcal{A}}$-sentence.
            \item $\mathcal{A}\Vdash_{p} \bot$, if and only if, $p=\infty$. Recall that $\bot$ is the symbol of contradiction.
        \end{enumerate}
    \end{enumerate}
\end{Def}

We propose now an extension of the forcing relation based on the ideas we exposed in the introduction of this subsection.

\begin{Def}[Modal Residuated Kripke forcing, \cite{Moncayo2023} Definition 3.1.35]
    Given $\mathcal{A}=(\mathbb{P}, \leq, \delta, \Vdash, D)$ a $MR$-Kripke $\mathcal{L}$-model, $p\in \mathbb{P}$ and an $MR-\mathcal{L}_{\mathcal{A}}$-sentence $\varphi$, we can extend the forcing relation $\mathcal{A}\Vdash_p \varphi$ to all $MR-\mathcal{L}_{\mathcal{A}}$-sentences by recursion on the complexity of $\varphi$. The definition of $\Vdash$ for the usual symbols is the same as in Definition \ref{ResiduatedKripkeForcing}, so we only define $\Diamond$.

    \begin{enumerate}
        \item $\mathcal{A}\Vdash_p \Diamond \varphi$, if and only if, there exists $q\in \mathbb{P}$ such that  $\mathcal{A}\Vdash_q \varphi$ and $\delta(q)\leq p$. 
    \end{enumerate}
    Notice that the definition of $\Diamond$ depends on $\delta$, so if we change the conucleus $\delta$, we would obtain different notions of possibility $\Diamond$.
\end{Def}

\begin{Lema}[\cite{Moncayo2023}, Lemma 3.1.36, cf. \cite{Ono1985} Lemma 2.1]
    Let $\varphi$ be an $MR-\mathcal{L}_{\mathcal{A}}$-sentence and $a_i, b\in \mathbb{P}$, with $i\in I$. If $\mathcal{A}\Vdash_{a_i}\varphi$ for every $i\in I$ and $\bigwedge\limits_{i\in I}a_i\leq b$, then $\mathcal{A}\Vdash_{b}\varphi$. Notice that this lemma implies that the set 
    \begin{center}
        $\{p\in\mathbb{P}: \mathcal{A}\Vdash_p \varphi\}$ 
    \end{center}
    is strongly hereditary.
\end{Lema}
\begin{proof}
    By Lemma \ref{sentencesstronglyhereditary}, we just need to see what happens with $\Diamond \varphi$ for a given $\varphi$ such that
    \begin{center}
        for all $p, p_i\in \mathbb{P}$ with $i\in I$, if $\mathcal{A}\Vdash_{p_i} \varphi$ and $\bigwedge\limits_{i\in I}p_i\leq p$, then $\mathcal{A}\Vdash_{p} \varphi$ (induction hypothesis).
    \end{center}
    
    Take $p, p_i\in \mathbb{P}$ with $i\in I$ such that $\mathcal{A}\Vdash_{p_i} \Diamond\varphi$
    and $\bigwedge\limits_{i\in I}p_i\leq p$. By definition of $\Vdash$, there exists $q_i\in \mathbb{P}$ such that
    \begin{center}
        $\delta(q_i)\leq p_i$ and $\mathcal{A}\Vdash_{q_i} \varphi$
    \end{center}
    By the induction hypothesis, we have $\mathcal{A}\Vdash_q \varphi$, where $q=\bigwedge\limits_{i\in I}q_i$. Notice that 
    \begin{align*}
            \delta(q)=\delta\left(\bigwedge\limits_{i\in I}q_i\right)&=\bigwedge\limits_{i\in I}\delta(q_i) \hspace{0.5cm}\mbox{(by Definition \ref{conucleusonorder} item 5.)} \\
            &\leq \bigwedge\limits_{i\in I}p_i  \hspace{1cm}\mbox{(since $\delta(q_i)\leq p_i$ for all }i\in I) \\
            & \leq p 
    \end{align*}
    Then, by definition of $\Vdash$, we conclude that $\mathcal{A}\Vdash_{p} \Diamond\varphi$
\end{proof}

\begin{Teo}[\cite{Moncayo2023}, Theorem 3.1.37.]\label{gamma*}
    The operation $\gamma_{\delta}:\mathbb{P}^*\rightarrow \mathbb{P}^*$ defined as 
    \begin{center}
        $\gamma_{\delta}(A):=\{p\in \mathbb{P}: \text{there is } q\in A \text{ such that } \delta(q)\leq p\}$
    \end{center}
    is a quantic nucleus on $(\mathbb{P}^*, \subseteq, \cdot)$.
    If there is no ambiguity, we denote $\gamma:=\gamma_{\delta}$.

\end{Teo}

\begin{Lema}[\cite{Moncayo2023}, Lemma 3.1.38]
    Let $A_i\in \mathbb{P}^*$ for $i\in I$. The quantic nucleus $\gamma$ satisfies
    \begin{center}
        $\gamma(\bigvee\limits_{i\in I}A_i)=\bigvee\limits_{i\in I}\gamma(A_i)$
    \end{center}
\end{Lema}

\begin{Obs}
        Notice that the previous lemma implies, by the Adjoint Functor Theorem for preorders, that $\gamma$ has a right adjoint. This adjoint may be used to define a notion of necessity just as we mentioned before.
\end{Obs}

\begin{Def}[\cite{Moncayo2023}, Definition 3.1.40.]
    Let $\delta:\mathbb{P}\rightarrow\mathbb{P}$ be a conucleus. 
    \begin{enumerate}
        \item $\delta$ is said to be \textit{idempotent} if $\delta(p^2):=\delta(p\cdot p)=\delta(p)$, for every $p\in \mathbb{P}$.\index{conucleus! idempotent}
        \item $\delta$ is said to \textit{respects the top element} if $\delta(\infty)=\infty$.\index{conucleus! respects the top element}
        \item $\delta$ is said to \textit{respect implications} if $\delta(p\rightarrow q)=1$, if and only if, $p\rightarrow \delta(q)=1$ for every $p, q\in \mathbb{P}$ (see Remark \ref{implicationonP} for the definition of $\rightarrow$).\index{conucleus! respect implications}        
    \end{enumerate}
\end{Def}

\begin{Teo}[\cite{Moncayo2023}, Theorem 3.1.41.]\label{gamma*isstandard}
    Let $\delta:\mathbb{P}\rightarrow\mathbb{P}$ be a conucleus. 
    \begin{enumerate}
        \item If $\delta$ is idempotent, then $\gamma$ is idempotent.        
        \item If $\delta$ respects the top element, then $\gamma$ respects the bottom element.
        \item If $\delta$ respects implications, then $\gamma$ respects implications.              
    \end{enumerate}
\end{Teo}
\begin{proof}

    Take $A, B\in \mathbb{P}^*$.
    
    \begin{enumerate}
        \item Assume that $\delta$ is idempotent. We want to see that $\gamma(A\cdot A)=\gamma(A)$. Since $(\mathbb{P}^*, \subseteq, \cdot)$ is a commutative integral quantale, by Theorem \ref{cuantales1} item 6, $A\cdot A\subseteq A$ and then, since $\gamma$ is monotone,
        \begin{center}
            $\gamma(A\cdot A)\subseteq \gamma(A)$.
        \end{center}
        Now take $p\in \gamma(A)$. Then, there exists $q\in A$ such that $\delta(q)\leq p$. Since $\delta$ is idempotent, $\delta(q\cdot q)=\delta(q)\leq p$. Therefore, since $q\cdot q\in A\cdot A$, we have that $p\in \gamma(A\cdot A)$, that is, 
         \begin{center}
             $\gamma(A)\subseteq \gamma(A\cdot A)$.
         \end{center}

        \item Recall that $0_{\mathbb{P}^*}=\{\infty\}$, therefore

        \begin{align*}
            \gamma(\{\infty\})&=\{p\in \mathbb{P}: \text{ there exists } q\in \{\infty\} (\delta(q)\leq p)\}=\{p\in \mathbb{P}: \delta(\infty)\leq p\} \\
            &=\{p\in \mathbb{P}: \infty\leq p\} \hspace{1cm} \mbox{(since $\delta$ respects the top element)}\\
            & =\{\infty\}
        \end{align*}
        
        \item Assume now that $\delta$ respects implications. Since $1_{\mathbb{P}^*}=\mathbb{P}$, we want to prove that 
        \begin{center}
             $\gamma(A\rightarrow B)=\mathbb{P}$, if and only if, $A\rightarrow \gamma(B)=\mathbb{P}$.
        \end{center}
        
        Since $\gamma$ is a quantic nucleus on $\mathbb{P}^*$, by Corollary \ref{gammaimplication}, we have that 
        \begin{center}
            $\gamma(A\rightarrow B)\subseteq A\rightarrow \gamma(B)$.
        \end{center}
        And therefore $\gamma(A\rightarrow B)=\mathbb{P}$ implies that $A\rightarrow \gamma(B)=\mathbb{P}$.

        Let us see converse. Let us assume that $A\rightarrow \gamma(B)=\mathbb{P}$.
        
        Take $c\in \mathbb{P}=A\rightarrow \gamma(B)$, that is, $c\cdot A\subseteq \gamma(B)$. This means that for every $a\in A$, there exists $b_a\in B$ such that $\delta(b_a)\leq c\cdot a$. Notice now that
        
        \begin{align*}
            \delta(b_a)\leq (c\cdot a)\cdot 1&\text{ iff } (c\cdot a)\rightarrow \delta(b_a)\leq 1 \hspace{0.5cm} &&\mbox{(by Theorem \ref{propertiesSOmonoidP} item 1.) } \\
            &\text{ iff } \delta((c\cdot a)\rightarrow b_a)=1\hspace{0.5cm} &&\mbox{(since $\delta$ respects implications)}\\
            & \text{ iff } \delta(c \rightarrow (a\rightarrow b_a))=1\hspace{0.5cm} &&\mbox{(by Theorem \ref{propertiesSOmonoidP} item 5.)} \\
            & \text{ iff } c\rightarrow \delta(a\rightarrow b_a)=1 \hspace{0.5cm} &&\mbox{(since $\delta$ respects implications)}\\
            & \text{ iff } \delta(a\rightarrow b_a)\leq c \hspace{0.5cm} &&\mbox{(by Theorem \ref{propertiesSOmonoidP} item 1.)}
    \end{align*}
    If we find some $a\in A$ such that $a\rightarrow b_a\in A\rightarrow B$, we get that $c\in \gamma(A\rightarrow B)$. Thus, take an enumeration $A=\{a_i:i\in I\}$ and define
    \begin{center}
        $a=\bigwedge\limits_{i\in I}a_i$.
    \end{center}
    Since $A$ is strongly hereditary, $a\in A$. Let us see that $a\rightarrow b_a\in A\rightarrow B$, that is, $(a\rightarrow b_a)\cdot A\subseteq B$. So take any $j\in I$ and let us show that $(a\rightarrow b_a)\cdot a_j\in B$.
    
    Since $a=\bigwedge\limits_{i\in I}a_i\leq a_j$, by Theorem \ref{propertiesSOmonoidP} item 3., we deduce that $a_j\rightarrow b_a\leq a\rightarrow b_a$.
    
    Therefore,
    
    \begin{align*}
            b_a&\leq (a_j\rightarrow b_a)\cdot a_j\hspace{0.5cm} \mbox{(by Theorem \ref{propertiesSOmonoidP} item 2.) }\\
	    &\leq (a\rightarrow b_a)\cdot a_j \hspace{0.7cm}\mbox{($\cdot$ is a monotone function)} 
    \end{align*}
    and since $B$ is strongly hereditary and $b_a\in B$, we deduce $(a\rightarrow b_a)\cdot a_j\in B$. Then, since $a\rightarrow b_a\in A\rightarrow B$, we conclude that $c\in \gamma(A\rightarrow B)$.

    \end{enumerate}
\end{proof}

\begin{Teo}\label{gammadiamond}
    Take $\mathcal{A}=(\mathbb{P}, \leq, \delta, \Vdash, D)$ a Modal Residuated Kripke $\mathcal{L}$-model, and an $MR-\mathcal{L}_{\mathcal{A}}$-sentence $\varphi$. Then,
    \begin{center}
        $\gamma(\{p\in \mathbb{P}:\mathcal{A}\Vdash_p \varphi\})=\{p\in \mathbb{P}:\mathcal{A}\Vdash_p \Diamond\varphi\}$
    \end{center}
\end{Teo}

\begin{proof}
    Take $p\in \mathbb{P}$, thus
    \begin{align*}
            r\in \{p\in \mathbb{P}:\mathcal{A}\Vdash_p \Diamond\varphi\}&\text{ iff } \text{ there is } q\in \mathbb{P} \mbox{ such that } \delta(q)\leq r \mbox{ and } \mathcal{A}\Vdash_q \varphi \hspace{1cm} \mbox{(by definition of $\Vdash$)} \\
            &\text{ iff } \text{ there is }  q\in \mathbb{P} \mbox{ such that } \delta(q)\leq r \mbox{ and } q\in \{p\in \mathbb{P}:\mathcal{A}\Vdash_p \varphi\}\\
            & \text{ iff } r\in\gamma(\{p\in \mathbb{P}:\mathcal{A}\Vdash_p \varphi\})\hspace{5.1cm} \mbox{(by definition of $\gamma$)} \\
    \end{align*}
\end{proof}

\section{Modal Residuated Kripke models of Set Theory}\label{ResiduatedKripkemodelsofsettheory}

The goal of this section is to generalize Fitting's translation results on Intuitionistic Kripke models from \cite{Fitting1969}, Chapter 14. Fitting constructs a sequence of Intuitionistic Kripke models $\mathcal{V}^{\mathbb{P}^+}_{\alpha}=(\mathbb{P}, \leq, \Vdash,  R_{\alpha}^{\mathbb{P}^+})$  that is ``isomorphic'' to a sequence of Boolean valued models $(R_{\alpha}^{\mathbb{B}}, \ldbrack \cdot\rdbrack_{\alpha}^\mathbb{B})$. In particular, we want to generalize the following double negation translation result:

    \begin{Coro}[\cite{Fitting1969}, Chapter 15, Corollary 5.6] \label{phiiffnotnotphi}
        If $\varphi$ is an $\mathcal{L}_{\in}$-sentence with no universal quantifiers, then $\varphi$ is valid in the model $R_{\alpha}^{\mathbb{B}}$ (that is $\ldbrack \varphi\rdbrack_{\alpha}^\mathbb{B}=1$), if and only if, $\lnot \lnot \varphi$ is valid in $(\mathbb{P}, \leq, \Vdash,  R_{\alpha}^{\mathbb{P}^+})$ (that is $R_{\alpha}^{\mathbb{P}^+}\Vdash_p \varphi$ for every $p\in \mathbb{P}$.)
    \end{Coro}
 
 In this section we will build a suitable generalization of Fitting model in the context of Residuated Logic (see Definition \ref{definitionresiduatedkripkehierarchy}) such that there exists an ``isomorphism'' (see Theorem \ref{isomorphismresiduatedclassic}) between it and an appropriate Heyting valued model (see Definition \ref{Heytingmodelsalafitting}). Furthermore,  we show in Corollary \ref{phiiffdiamondphionmodels} that if $\varphi$ is an $MR-\mathcal{L}_{\in}$-sentence without universal quantifiers that is valid in the Heyting model given in Definition \ref{Heytingmodelsalafitting}, then $\Diamond \varphi$ is valid in the Residuated Kripke model given in Definition \ref{definitionresiduatedkripkehierarchy}.

We show now the general structure of this subsection and how it relates to Fitting's original construction. We start by noticing that the double negation operator $\lnot \lnot$ is a modal (closure) operator on a Heyting algebra, and it is used to translate sentences (Corollary \ref{phiiffnotnotphi}) from the Boolean valued model $R^{\mathbb{B}}$ into the Intuitionistic Kripke model $\mathcal{V}^{\mathbb{P}^+}$. Fitting starts with an Intuitionistic Kripke model with constant universe $\mathcal{A}=(\mathbb{P}, \leq, \Vdash, D)$. The key points to notice are as follows: 
\begin{enumerate}
    \item The set $\mathbb{P}^+=\{A\subseteq \mathbb{P}: \text{A is hereditary}\}$ is a Heyting algebra.
    \item The operator double negation $\lnot \lnot$ is a modal operator on a Heyting algebra.
    \item The set $\mathcal{F}_{\lnot \lnot}:=\{x\in \mathbb{P}^+:\lnot \lnot x=1\}$ is a filter on $\mathbb{P}^+$ such that $\mathbb{P}^+/\mathcal{F}_{\lnot \lnot}$ is a Boolean algebra.
    \item $\lnot \lnot$ is used to translate sentences from the Boolean valued model into the Intuitionistic Kripke models (see Corollary \ref{phiiffnotnotphi} ).    
\end{enumerate}

Therefore, one could wonder if this kind of results can be obtain by starting with some $MR$-Kripke $\mathcal{L}$-model $\mathcal{A}=(\mathbb{P}, \leq, \delta, \Vdash, D)$ and then finding a suitable valued model such that an analogous of Corollary \ref{phiiffnotnotphi} holds. \\

Throughout this subsection, fix a Modal Residuated Kripke $\mathcal{L}$-model $\mathcal{A}=(\mathbb{P}, \leq, \delta, \Vdash, D)$ with $\delta:\mathbb{P}\rightarrow \mathbb{P}$ an idempotent conucleus that respects implications and the top element.

\begin{Def}[\cite{Moncayo2023}, Definition 3.2.15, cf. \cite{Fitting1969}, p. 166]\index{$\mathbb{P}^*$!-subset}\label{residuatedPsubset}
    We say that a function $f$ is a \textit{$\mathbb{P}^*$-subset of $\mathcal{A}$} if
    \begin{enumerate}
        \item $Dom(f)\subseteq D$
        \item $Ran(f)\subseteq \mathbb{P}^*=\{A\subseteq \mathbb{P}: A \text{ is strongly hereditary}\}$
    \end{enumerate}
    Recall Theorem \ref{P*isaresiduatedlattice} says that $\mathbb{P}^*$ is a Residuated Lattice with the order $\subseteq$ and the product
    \begin{center}
     $A\cdot B=\{c\in \mathbb{P}: \text{ there exist }a\in A, b\in B \text{ such that } c\geq a\cdot b \}$, where $A, B\in \mathbb{P}^*$  
    \end{center}
     and that Theorems \ref{gamma*} and \ref{gamma*isstandard} state that $\gamma:\mathbb{P}^*\rightarrow \mathbb{P}^*$ defined by 
    \begin{center}
         $\gamma(A):=\{p\in \mathbb{P}: \exists q\in A (\delta(q)\leq p)\}$
    \end{center}
    is an idempotent quantic nucleus that respects implications and the bottom element, that is, $\gamma$ is a standard quantic nucleus on $\mathbb{P^*}$.
     
\end{Def}

\begin{Def}[\cite{Moncayo2023}, Definition 3.2.16,  cf. \cite{Fitting1969}, p. 166]\index{$\gamma$!-regular}
    Let $\mathbb{Q}$ be any complete Residuated Lattice and $\gamma$ be any standard quantic nucleus on $\mathbb{Q}$. We call an element $x\in\mathbb{Q}$ \textit{$\gamma$-regular} if $\sim\sim\gamma(x)=x$. This definition generalizes the notion of \textbf{regular} sets in a Heyting algebra. We call a function with range $\mathbb{Q}$ \textit{$\gamma$-regular}, if every member of its range is $\gamma$-regular. 
\end{Def}

\begin{Def}[\cite{Moncayo2023}, Definition 3.2.17, cf. \cite{Fitting1969}, p. 166]\index{extensional function}
    We say that a function from $D$ to $\mathbb{P}^*$ is \textit{extensional} if, for each $g, h\in D$
    \begin{center}
        $f(g)\cdot \{p\in \mathbb{P}: \mathcal{A}\Vdash_p (g=h)\}\subseteq f(h)$
    \end{center}
    where $(g=h)$ is an abbreviation defined by:
    \begin{center}
        $(g=h):=\Diamond\sim(\exists x)\sim (x\in g\rightarrow x\in h))\& (\Diamond\sim(\exists x)\sim(x\in h\leftarrow x\in g))$.
    \end{center}    
    And we denote 
    \begin{center}
        $\mathcal{P}^{\mathbb{P}^*}(D):=\{f: f \text{ is a $\gamma$-regular and extensional $\mathbb{P}^*$-subset of }\mathcal{A}\}$
    \end{center}
\end{Def}

\begin{Def}[\cite{Moncayo2023}, Definition 3.2.18, cf. \cite{Fitting1969}, p. 166]\label{definitionresiduatedkripkehierarchy}
    We now define on induction on ordinals a class of $MR$-Kripke $\mathcal{L}_{\in}$-models $\mathcal{V}^{\mathbb{P}^*}_{\alpha}:=(\mathbb{P}, \leq,  \delta, \Vdash,R^{\mathbb{P}^*}_{\alpha})$ all with the same underlying order $(\mathbb{P},\leq)$ but changing the universe for each ordinal $\alpha$ as follows:
    \begin{enumerate}
    
        \item $\mathcal{V}^{\mathbb{P}^*}_0:=(\mathbb{P}, \leq,  \delta, \Vdash,R^{\mathbb{P}^*}_0)$ where $R^{\mathbb{P}^*}_0:=\emptyset$.
        
        \item $\mathcal{V}^{\mathbb{P}^*}_{\alpha+1}:=(\mathbb{P}, \leq,  \delta, \Vdash,R^{\mathbb{P}^*}_{\alpha+1})$ where $R^\mathbb{P^*}_{\alpha+1}:=R^{\mathbb{P}^*}_{\alpha}\cup \mathcal{P}^{\mathbb{P}^*}(R^\mathbb{P^*}_{\alpha})$ and $\mathcal{V}^{\mathbb{P}^*}_{\alpha+1}\Vdash_p f\in g $ is defined as follows: \\
        If $p\in \mathbb{P}$ and $f, g\in R^{\mathbb{P}^*}_{\alpha+1}$ then we have the following cases:
        \begin{enumerate}[label=\alph*.]
            \item If $f, g\in R^{\mathbb{P}^*}_{\alpha}$, then 
            \begin{center}
                $\mathcal{V}^{\mathbb{P}^*}_{\alpha+1}\Vdash_p (f\in g)$, if and only if, $\mathcal{V}^{\mathbb{P}^*}_{\alpha}\Vdash_p (f\in g)$.
            \end{center}
            
            \item If $f\in R^{\mathbb{P}^*}_{\alpha}$ and $g\in R^{\mathbb{P}^*}_{\alpha+1}\setminus R^{\mathbb{P}^*}_{\alpha}=\mathcal{P}^{\mathbb{P}^*}(R^{\mathbb{P}^*}_{\alpha})$, then 
            \begin{center}
                $\mathcal{V}^{\mathbb{P}^*}_{\alpha+1}\Vdash_p (f\in g)$, if and only if, $p\in g(f)$.
            \end{center}
            
            \item If $f\in R^{\mathbb{P}^*}_{\alpha+1}\setminus R^{\mathbb{P}^*}_{\alpha}=\mathcal{P}^{\mathbb{P}^*}(R^{\mathbb{P}^*}_{\alpha})$, then $\mathcal{V}^{\mathbb{P}^*}_{\alpha+1}\Vdash_p (f\in g)$, if and only if,
            
            \begin{center}               
                $p\in \bigvee\limits_{h\in dom(g)}P_h$
            \end{center}
            where            
           $$P_h:=g(h)\cdot \left(P_{f\subseteq h}\cdot P_{h\subseteq f}\right)$$    
           and
            \begin{align*}
                P_{f\subseteq h}:=
                 \bigcap\limits_{x\in R^{\mathbb{P}^*}_{\alpha}}(f(x)\rightarrow \{q\in \mathbb{P}: \mathcal{V}^{\mathbb{P}^*}_{\alpha} \Vdash_q \sim\sim\Diamond (x\in h)\})\\
                 P_{h\subseteq f}:=
                 \bigcap\limits_{x\in R^{\mathbb{P}^*}_{\alpha}}(f(x)\leftarrow \{q\in \mathbb{P}: \mathcal{V}^{\mathbb{P}^*}_{\alpha} \Vdash_q \sim\sim\Diamond (x\in h)\}).
            \end{align*}
        \end{enumerate}
        \item If $\alpha\not=0$ is a limit ordinal, then let $R_{\alpha}^{\mathbb{P}^*}:=\bigcup\limits_{\beta<\alpha} R_{\beta}^{\mathbb{P^*}}$ and given $f, g\in R_{\alpha}^{\mathbb{P}^*}$ take any $\eta<\alpha$ such that $f, g\in R_{\eta}^{\mathbb{P}}$ and let 
        \begin{center}
            $\mathcal{V}^{\mathbb{P}^*}_{\alpha}\Vdash_p (f\in g)$, if and only if, $\mathcal{V}^{\mathbb{P}^*}_{\eta}\Vdash_p (f\in g)$.
        \end{center}
    \end{enumerate}
    For every $\alpha$ we also define
    \begin{center}
        $\mathcal{V}^{\mathbb{P}^*}_{\alpha}\Vdash_{p} \bot$, if and only if, $p=\infty$.
    \end{center}
\end{Def}

\begin{Obs}[\cite{Moncayo2023}, Remark 3.2.19, cf. \cite{Fitting1969}, Remark 4.2]
        The expression  
        \begin{center}
                $\scalemath{0.9}
                {\bigcap\limits_{x\in R^{\mathbb{P}^*}_{\alpha}}(f(x)\rightarrow \{q\in \mathbb{P}: \mathcal{V}^{\mathbb{P}^*}_{\alpha} \Vdash_q \sim \sim\Diamond (x\in h)\})\cdot\bigcap\limits_{x\in R^{\mathbb{P}^*}_{\alpha}}(f(x)\leftarrow \{q\in \mathbb{P}: \mathcal{V}^{\mathbb{P}^*}_{\alpha} \Vdash_q \sim \sim\Diamond (x\in h)\})}$
        \end{center}       
        
        is an element in the Residuated Lattice $\mathbb{P}^*$, where $\cdot, \rightarrow, \leftarrow $ and $\bigcap$ are the operations on $\mathbb{P}^*$ as a Residuated Lattice (see Definition \ref{operationsstronglyhereditaty}).
\end{Obs}

\begin{Def}[\cite{Moncayo2023}, Definition 3.2.20, cf. \cite{Fitting1969}, p. 166]\label{V^P}
    Consider the $MR$-Kripke (class) $\mathcal{L}_{\in}$-model 
    \begin{center}
        $\mathcal{V}^{\mathbb{P}^*}:=(\mathbb{P}, \leq,  \delta, \Vdash, R^{\mathbb{P}^*})$, where $R^{\mathbb{P}^*}:=\bigcup\limits_{\alpha\in ON} R_{\alpha}^{\mathbb{P}^*}$ 
    \end{center}    
    and given $f, g\in R^{\mathbb{P}^*}$, take any $\eta\in ON$ such that $f, g\in R_{\eta}^{\mathbb{P}}$ and define 
    \begin{center}
        $\mathcal{V}^{\mathbb{P}^*}\Vdash_p (f\in g)$, if and only if, $\mathcal{V}^{\mathbb{P}^*}_{\eta}\Vdash_p (f\in g)$.
    \end{center}
\end{Def}

We need to see that this definition provides indeed a Modal Residuated Kripke model, that is, it satisfies the Definition \ref{ModalResiduatedKripkeModel} item 3. sub-items a., b. and c.. By definition of $\Vdash$, the model $\mathcal{V}^{\mathbb{P}^*}_{\alpha}$ satisfies condition c. and since $\infty$ is an element of every strongly hereditary set (see Definition \ref{stronglyhereditary}) it is straightforward to see that condition b. also holds. Therefore, we just need to check that a. holds.

\begin{Lema}
    [\cite{Moncayo2023}, Theorem 3.2.21.]
    For every $\alpha\in ON$, we have that if $p_i, q\in \mathbb{P}$ with $i\in I$,
    \begin{center}
        if $\mathcal{V}^{\mathbb{P}^*}_{\alpha}\Vdash_{p_i} (f\in g)$ for every $i\in I$ and $\bigwedge\limits_{i\in I} p_i\leq q$, then $\mathcal{V}^{\mathbb{P}^*}_{\alpha}\Vdash_{q} (f\in g)$.
    \end{center}
\end{Lema}
\begin{proof}
    We prove this by transfinite induction. Since the cases for $\alpha=0$ and $\alpha$ a limit ordinal are trivial, we only consider what happens at the successor step, so let us suppose the following holds at $\alpha$:
    \begin{center}
        if $\mathcal{V}^{\mathbb{P}^*}_{\alpha}\Vdash_{p_i} (f\in g)$ for $i\in I$ and $\bigwedge\limits_{i\in I} p_i\leq q$, then $\mathcal{V}^{\mathbb{P}^*}_{\alpha}\Vdash_{q} (f\in g)$ (induction hypothesis),
    \end{center}
    and let us prove it at $\alpha+1$. Assume that 
    \begin{center}        $\mathcal{V}^{\mathbb{P}^*}_{\alpha+1}\Vdash_{p_i} (f\in g)$ for $i\in I$ and $\bigwedge\limits_{i\in I} p_i\leq q$.
    \end{center}
    We have three cases:
    \begin{enumerate}
            \item If $f, g\in R^{\mathbb{P}^*}_{\alpha}$, then we have the result by the induction hypothesis.
            
            \item If $f\in R^{\mathbb{P}^*}_{\alpha}$ and $g\in R^{\mathbb{P}^*}_{\alpha+1}\setminus R^{\mathbb{P}^*}_{\alpha}=\mathcal{P}^{\mathbb{P}^*}(R^{\mathbb{P}^*}_{\alpha})$, then
            \begin{center}                $\mathcal{V}^{\mathbb{P}^*}_{\alpha+1}\Vdash_{p_i} (f\in g)$ means that $p_i\in g(f)$.
            \end{center}
            but by definition of $g\in\mathcal{P}^{\mathbb{P}^*}(R^{\mathbb{P}^*}_{\alpha})$ we know that the codomain of $g$ is $\mathbb{P}^*$ and so $g(f)$ is strongly hereditary. Therefore, $q\in g(f)$ and thus 
            \begin{center}                $\mathcal{V}^{\mathbb{P}^*}_{\alpha+1}\Vdash_{q} (f\in g)$
            \end{center}

            \item If $f\in R^{\mathbb{P}^*}_{\alpha+1}\setminus R^{\mathbb{P}^*}_{\alpha}=\mathcal{P}^{\mathbb{P}^*}(R^{\mathbb{P}^*}_{\alpha})$, then
            \begin{center}                $\mathcal{V}^{\mathbb{P}^*}_{\alpha+1}\Vdash_{p_i} (f\in g)$ means that $p_i\in \bigvee\limits_{h\in dom(g)}P_h=\bigvee\limits_{h\in dom(g)}g(h)\cdot \left(P_{f\subseteq h}\cdot P_{h\subseteq f}\right)$
            \end{center}

            but 

            \begin{align*}          
                P_{f\subseteq h}:=&\bigcap\limits_{x\in R^{\mathbb{P}^*}_{\alpha}}(f(x)\rightarrow \{q\in \mathbb{P}: \mathcal{V}^{\mathbb{P}^*}_{\alpha} \Vdash_q \sim \sim\Diamond (x\in h)\})\\
                &=\bigcap\limits_{x\in R^{\mathbb{P}^*}_{\alpha}}(f(x)\rightarrow \sim\sim\gamma\{q\in \mathbb{P}: \mathcal{V}^{\mathbb{P}^*}_{\alpha} \Vdash_q  (x\in h)\}) \hspace{0.3cm} \mbox{(by Theorems }  \ref{gammadiamond} \mbox{ and } \ref{kripkecongruence}).
            \end{align*}
            By the induction hypothesis, the set $\{q\in \mathbb{P}: \mathcal{V}^{\mathbb{P}^*}_{\alpha} \Vdash_q  (x\in h)\}$ is strongly hereditary, and since $f\in \mathcal{P}^{\mathbb{P}^*}(R^{\mathbb{P}^*}_{\alpha})$, $f(x)$ is also strongly hereditary for every $x\in dom(g)$. Therefore, since the operations $\sim, \rightarrow, \cdot, \gamma$ and $\bigcap$ are all closed in $\mathbb{P}^*$ (see Theorems \ref{P*isaresiduatedlattice} and \ref{gamma*}), we have that $P_{f\subseteq h}\in \mathbb{P}^*$. By using a similar argument, we can show that $P_{h\subseteq f}\in \mathbb{P}^*$  and since $g(h)\in \mathbb{P}^*$, we have that $\bigvee\limits_{h\in dom(g)}P_h$ is strongly hereditary. Therefore, $q\in\bigvee\limits_{h\in dom(g)}P_h$ and thus 
            \begin{center}                  $\mathcal{V}^{\mathbb{P}^*}_{\alpha+1}\Vdash_{q} (f\in g)$
            \end{center}
    \end{enumerate}
\end{proof}

We now construct a Heyting valued model $(R^{\mathbb{H}}, \ldbrack \cdot\rdbrack^\mathbb{H})$ that is related to\linebreak  $\mathcal{V}^{\mathbb{P}^*}=(\mathbb{P}, \leq, R^{\mathbb{P}^*})$.

\begin{Def}[\cite{Moncayo2023}, Definition 3.2.22, cf. \cite{Fitting1969}, p. 164]
    Given a Heyting valued model $(R, \ldbrack \cdot\rdbrack^\mathbb{H})$, we say that a function $f:R\rightarrow \mathbb{H}$ is \textit{extensional} if for all $g, h\in R$,
    \begin{center}
        $f(g)\wedge \ldbrack \lnot(\exists x)\lnot(x\in g \rightarrow x\in h)\rdbrack^\mathbb{H}\land \ldbrack \lnot(\exists x)\lnot(x\in g \leftarrow x\in h)\rdbrack^\mathbb{H}\leq f(h)$
    \end{center}
\end{Def}

and we say that a function $f:R\rightarrow \mathbb{H}$ is \textit{regular} if $\lnot \lnot f(x)=f(x)$ for every $x\in R$.

\begin{Def}[\cite{Moncayo2023}, Definition 3.2.23, cf. \cite{Fitting1969}, p. 165]\label{Heytingmodelsalafitting}
    We now define on induction over the ordinals a class of Heyting valued models $(R_{\alpha}^{\mathbb{H}}, \ldbrack \cdot\rdbrack_{\alpha}^\mathbb{H})$
    \begin{enumerate}
    
        \item $R_{0}^{\mathbb{H}}:=\emptyset$ with $\ldbrack \cdot\rdbrack_{0}^\mathbb{H}:=\emptyset$.
        
        \item
        $R_{\alpha+1}^{\mathbb{H}}:=R_{\alpha}^{\mathbb{H}}\cup\{f:R_{\alpha}^{\mathbb{H}}\rightarrow \mathbb{H}: f \text{ is extensional and regular}\}$ and given $f, g\in R^\mathbb{H}_{\alpha+1}$, we have, for the definition of  $\ldbrack f\in g\rdbrack_{\alpha}^\mathbb{H}$, the following cases:
        \begin{enumerate}[label=\alph*.]
            \item If $f, g\in R^\mathbb{H}_{\alpha}$, define  $\ldbrack f\in g\rdbrack_{\alpha+1}^\mathbb{H}:=\ldbrack f\in g\rdbrack_{\alpha}^\mathbb{H}$.
            
            \item If $f\in R^\mathbb{H}_{\alpha}$ and $g\in R^\mathbb{H}_{\alpha+1}\setminus R^{\mathbb{H}}_{\alpha}$, define $\ldbrack f\in g\rdbrack_{\alpha+1}^\mathbb{H}:=g(f)$.
            
            \item If $f\in R^{\mathbb{H}}_{\alpha+1}\setminus R^{\mathbb{H}}_{\alpha}$, define
            \begin{center}
                $\ldbrack f\in g\rdbrack_{\alpha+1}^\mathbb{H}:= \bigvee\limits_{h\in dom(g)}\{g(h)\wedge \bigwedge\limits_{x\in R_{\alpha}^{\mathbb{H}}}(f(x)\leftrightarrow \ldbrack \lnot\lnot(x\in h)\rdbrack_{\alpha}^\mathbb{H})\}$
            \end{center}
            
            \item If $\alpha\not=0$ is a limit ordinal, then let $R_{\alpha}^{\mathbb{H}}:=\bigcup\limits_{\beta<\alpha} R_{\beta}^{\mathbb{H}}$ and given $f, g\in R_{\alpha}^{\mathbb{H}}$ take any $\eta<\alpha$ such that $f, g\in R_{\eta}^{\mathbb{H}}$ and let $\ldbrack f\in g\rdbrack_{\alpha}^\mathbb{H}:=\ldbrack f\in g\rdbrack_{\eta}^\mathbb{H}$.
        \end{enumerate}
        Now let 
        \begin{center}            $R^{\mathbb{H}}:=\bigcup\limits_{\alpha\in ON} R_{\alpha}^{\mathbb{H}}$ 
        \end{center}
        and given $f, g\in R^{\mathbb{H}}$ take any $\eta<\alpha$ such that $f, g\in R_{\eta}^{\mathbb{H}}$ and let $\ldbrack f\in g\rdbrack^\mathbb{H}:=\ldbrack f\in g\rdbrack_{\eta}^\mathbb{H}$.
    \end{enumerate}
\end{Def}

\begin{Obs} 
    The construction given above mimics Fitting's Definition, but with two main differences:
    \begin{enumerate}
        \item Instead of using a Boolean algebra, we consider a Heyting algebra.
        \item In condition 2. c., we consider the term $\ldbrack \lnot\lnot(x\in h)\rdbrack_{\alpha}^\mathbb{H}$ rather than $\ldbrack x\in h\rdbrack_{\alpha}^\mathbb{H}$. Clearly, in the Classical (Boolean) case, these expressions are equivalent, but in the Intuitionistic case they are not.
    \end{enumerate}
\end{Obs}

\begin{Def}[\cite{Moncayo2023}, Definition 3.2.25.]\index{$\gamma$!-dense}
     Let $\mathbb{Q}$ be any complete residuated lattice and $\gamma$ be any quantic nucleus on $\mathbb{Q}$. An element of $x\in \mathbb{Q}$ is called \textit{$\gamma$-dense} if $\gamma(x)=1_{\mathbb{Q}}$. 
\end{Def}

\begin{Obs}
    The definition given above generalizes the notion of \textbf{dense} sets in a Heyting algebra. We focus on the case where $\mathbb{Q}=\mathbb{P}^*$ and $\gamma$ is the quantic nucleus determined by $\delta$. Let $\mathcal{F}_{\gamma}$ be the collection of all $\gamma$-dense elements of $\mathbb{P}^*$. By Theorems \ref{condensedisfilter} and \ref{FgammaisHeyting} $\mathcal{F}_{\gamma}$ is a filter such that $\mathbb{P}^*/\mathcal{F}_{\gamma}$ is a  Heyting  algebra.
\end{Obs}

    \begin{Obs}\label{RemarkoperationsonH}
         Recall (see Definition \ref{Q/F} and Theorem \ref{FgammaisHeyting}) that the relation $\approx_{\mathcal{F}_{\gamma}}$ on $\mathbb{P}^*$ given by 
    \begin{center}
        $A\approx_{\mathcal{F}_{\gamma}} B$, if and only if, $A\rightarrow B\in \mathcal{F}_{\gamma}$ and $B\rightarrow A\in \mathcal{F}_{\gamma}$
    \end{center}
    is an equivalence relation. Also, we have that
    \begin{center}        $\mathbb{H}:=\mathbb{P}^*/\mathcal{F}_{\gamma}=\mathbb{P}^*/_{\approx_{\mathcal{F}_{\gamma}}}=\{|A|: A\in \mathbb{P}^*\}$
    \end{center}
     is a complete Heyting algebra. Furthermore, if $|A|, |B|\in \mathbb{H}$
    \begin{center}
        $|A|\leq |B|$ iff $A\rightarrow B\in \mathcal{F}_{\gamma}$\\
        $|A|\land |B|=|A\land B|=|A\cdot B|=|A|\cdot |B|$\\
        $|A|\lor |B|=|A\lor B|$\\
        $|A|\rightarrow |B|=|A\rightarrow B|$\\
        $|\sim A|=\lnot|A|$\\
        $|A|=|\gamma(A)|$ (see Corollary \ref{clasedegammaxesclasedex})\\
        $|\bigvee\limits_{i\in I}A_i|=\bigvee\limits_{i\in I}|A_i|$ (see Theorem \ref{unionclass})
    \end{center}   
    \end{Obs}

    \begin{Obs}[cf. \cite{Fitting1969}, Remark 5.1]
        The equality $|\bigwedge\limits_{i\in I}A_i|=\bigwedge\limits_{i\in I}|A_i|$ is not true in general, and thus $MR$-formulas with universal quantifiers behave poorly (since, in valued models, we usually interpret universal quantifiers as meets). This explains why we do not consider formulas with universal quantifiers in the following theorem.
    \end{Obs}

   The Heyting algebra $\mathbb{H}:=\mathbb{P}^*/\mathcal{F}_{\gamma}$ determines a sequence of valued models $(R_{\alpha}^{\mathbb{H}}, \ldbrack \cdot\rdbrack_{\alpha}^\mathbb{H})$ that is isomorphic to the sequence $\mathcal{V}^{\mathbb{P}^*}_{\alpha}=(\mathbb{P}, \leq,  \delta, \Vdash, R_{\alpha}^{\mathbb{P}^*})$ in the following sense:
    
    \begin{Teo}[\cite{Moncayo2023}, Theorem 3.2.29, cf. \cite{Fitting1969}, Chapter 15, Theorem 5.5]\label{isomorphismresiduatedclassic}
        For every $\alpha\in ON$, there exist a bijection between $R_{\alpha}^{\mathbb{P}^*}$ and $R_{\alpha}^{\mathbb{H}}$ (where if $f\in R_{\alpha}^{\mathbb{P}^*}$, $f'$ denotes the image of $f$ via this bijection) such that for every MR-$\mathcal{L}_{\in}$-formula with no universal quantifiers $\varphi(x_1, ..., x_n)$ and every $a_1, ..., a_n\in R_{\alpha}^{\mathbb{P}^*}$,
        \begin{center}
            $|\{p\in \mathbb{P}: \mathcal{V}^{\mathbb{P}^*}_{\alpha}\Vdash_p \varphi(a_1, ..., a_n)\}|=\ldbrack \varphi(a'_1, ..., a'_n)\rdbrack_{\alpha}^\mathbb{H}$
        \end{center}
    \end{Teo}
    
    \begin{proof}
        We show this by induction on $\alpha$. $R_{0}^{\mathbb{P}^*}$ and $ R_{0}^{\mathbb{H}}$ are the same, so it holds for $0$.
        
        Assume that there exists such a bijection $'$ between $R_{\alpha}^{\mathbb{P}^*}$ and $R_{\alpha}^{\mathbb{H}}$ (induction hypothesis 1).
        
        Take $g\in R_{\alpha+1}^{\mathbb{P}^*}\setminus R_{\alpha}^{\mathbb{P}^*}=\mathcal{P}^{\mathbb{P}^*}(R_{\alpha}^{\mathbb{P}^*})$ and define the function $g': R_{\alpha}^{\mathbb{H}}\rightarrow \mathbb{H}$ in the following way: Since $':R_{\alpha}^{\mathbb{P}^*}\rightarrow R_{\alpha}^{\mathbb{H}}$ is onto, every element $F\in R^{\mathbb{H}}_{\alpha}$ has the form $F=f'$ for some $f\in R^{\mathbb{P}^*}_{\alpha}$. Therefore, we can define $g'$ by
        \begin{center}
            $g'(f'):=|g(f)|\in \mathbb{H}=\mathbb{P}^*/\mathcal{F}_{\gamma}$ for every $f'\in R^{\mathbb{H}}_{\alpha} $.
        \end{center}

        For now we assume that $g$ is extensional if and only if $g'$ is extensional. This will be proved at the end of this theorem.
        
        Let us see that this map is injective. Take $g, h\in R_{\alpha+1}^{\mathbb{P}^*}\setminus R_{\alpha}^{\mathbb{P}^*}=\mathcal{P}^{\mathbb{P}^*}(R_{\alpha}^{\mathbb{P}^*})$ such that $|g(f)|=|h(f)|$ for every $f\in R_{\alpha}^{\mathbb{P}^*}$. Then, by definition of $=$ in $\mathbb{P}^*/\mathcal{F}_{\gamma}$, we have, in particular
        
        \begin{center}
            $g(f)\rightarrow h(f)\in \mathcal{F}_{\gamma}$
        \end{center}
        that is, by definition of $\mathcal{F}_{\gamma}$, 
        \begin{center}
            $\gamma(g(f)\rightarrow h(f))=1$
        \end{center}
        and since by Corollary \ref{gammaimplication}, $\gamma(g(f)\rightarrow h(f))\leq g(f)\rightarrow \gamma(h(f))$, we have that
        \begin{center}
            $g(f)\rightarrow \gamma(h(f))=1$
        \end{center}
        By Theorem \ref{cuantales2} item 2.,  $\gamma(h(f))\leq (\sim\sim \gamma(h(f)))$ and by Theorem \ref{cuantales1} item 9.,
        \begin{center}
            $g(f)\rightarrow (\sim\sim\gamma(h(f)))=1$
        \end{center}        
        But $h$ is a $\gamma$-regular function, so
        \begin{center}
            $g(f)\rightarrow h(f)=1$
        \end{center}
        which implies, by Theorem \ref{cuantales1} item 1., that $g(f)\subseteq h(f)$. In a similar fashion we can show that $h(f)\subseteq g(f)$ and therefore we conclude $g(f)=h(f)$ for every $f\in R^{\mathbb{P}^*}_{\alpha}$, that is, $g=h$.
        
        To see that the map is surjective, let us take $h\in R_{\alpha+1}^{\mathbb{H}}\setminus R_{\alpha}^{\mathbb{H}}=\mathcal{P}^{\mathbb{H}}(R_{\alpha}^{\mathbb{H}})$, that is, $h:R_{\alpha}^{\mathbb{H}}\rightarrow \mathbb{H}$ is a regular and extensional function. We will construct a function $g\in \mathcal{P}^{\mathbb{P}^*}(R_{\alpha}^{\mathbb{P}^*})$ such that $g'=h$, that is, $g'(f')=h(f')$ for every $f'\in R_{\alpha}^{\mathbb{H}}$. Let $s$ be any function from $R_{\alpha}^{\mathbb{P}^*}$ to $\mathbb{P}^*$ such that
        
        \begin{center}
            for $f\in R_{\alpha}^{\mathbb{P}^*}$, $s(f)$ is a representative of the class $h(f')\in \mathbb{H}=\mathbb{P}^*/\mathcal{F}_{\gamma}$,
        \end{center}
        
        that is, $h(f')=|s(f)|$. Let $g$ be the function defined by
        \begin{center}
            $g(f)=\sim\sim\gamma(s(f))$ for $f\in R_{\alpha}^{\mathbb{P}^*}$.
        \end{center} 
        Then, by Theorem \ref{doublenegationgammaisidempotent}, $g$ is $\gamma$-regular and since its domain is $R_{\alpha}^{\mathbb{P}^*}$, we have that $g\in R_{\alpha+1}^{\mathbb{P}^*}\setminus R_{\alpha}^{\mathbb{P}^*}=\mathcal{P}^{\mathbb{P}^*}(R_{\alpha}^{\mathbb{P}^*})$. We want to see that $g'=h$, so let us take $F\in R_{\alpha}^{\mathbb{H}}$,  by the surjectiveness of $'$ there exists $f\in R^{\mathcal{P}^*}_\alpha$ such that $F=f'$. So, 
        
        \begin{align*}
            g'(f')&= |g(f)| \hspace{0.5cm} &&\mbox{(by definition of $g'$)} \\
            &=|\sim\sim\gamma (s(f))|\hspace{0.5cm} &&\mbox{(by definition of $g$)} \\
            &= \lnot\lnot|\gamma s(f)| \hspace{0.5cm} &&\mbox{(by  Remark~\ref{RemarkoperationsonH}} 
            \\&= \lnot\lnot|s(f)| \hspace{0.5cm} &&\mbox{(by  Corollary \ref{clasedegammaxesclasedex})} \\
            &= \lnot\lnot h(f')\hspace{0.5cm} &&\mbox{($s(f)$ is a representative of the class $h(f')$) }\\
            &= h(f')\hspace{0.5cm} &&\mbox{($h$ is regular function)}.
        \end{align*}
        
        Then, $g'=h$ and the function $'$ is surjective.
        
        By the induction hypothesis 1, we may assume that for every $MR-\mathcal{L}_{\in}$-formula with no universal quantifiers $\varphi(x_1, ..., x_n)$ and every $a_1, ..., a_n\in R_{\alpha}^{\mathbb{P}^*}$,
        
        \begin{center}
            $|\{p\in \mathbb{P}: \mathcal{V}^{\mathbb{P}^*}_{\alpha}\Vdash_p \varphi(a_1, ..., a_n)\}|=\ldbrack \varphi(a'_1, ..., a'_n)\rdbrack_{\alpha}^\mathbb{H}$
        \end{center}
        
        We will show that this result also holds for $R_{\alpha+1}^{\mathbb{P}^*}$ by induction on formulas. We start with the atomic case. Let $f, g\in R_{\alpha+1}^{\mathbb{P}^*}$. We have three cases:
        
        \begin{enumerate}
            \item If $f, g\in R_{\alpha}^{\mathbb{P}^*}$, then we have the result by the induction hypothesis 1.
            
            \item If $f\in R_{\alpha}^{\mathbb{P}^*}$ and $g\in R_{\alpha+1}^{\mathbb{P}^*}\setminus R_{\alpha}^{\mathbb{P}^*}$, then            
            \begin{align*}
                \ldbrack f'\in g'\rdbrack_{\alpha+1}^\mathbb{H}&= g'(f') \hspace{0.5cm} &&\mbox{(by definition of $\ldbrack \cdot \in \cdot\rdbrack_{\alpha+1}^\mathbb{H}$)} \\
                 &=|g(f)|\hspace{0.5cm} &&\mbox{(by definition of $g'$)} \\
                &= |\{p\in\mathbb{P}: p\in g(f)\}| \hspace{0.5cm} &&\mbox{(by definition of $g(f)$)}\\
                &=|\{p\in\mathbb{P}:\mathcal{V}_{\alpha+1}^{\mathbb{P}}\Vdash_p f\in g\}| \hspace{0.5cm} \small &&\mbox{(by definition of $\mathcal{V}_{\alpha+1}^{\mathbb{P}^*}\Vdash_p f\in g$)}
            \end{align*}
            
            \item If $f\in R_{\alpha+1}^{\mathbb{P}^*}$, recall that we denote 
            
            \begin{align*}
                &P_h:=g(h)\cdot \left(P_{f\subseteq h}\cdot P_{h\subseteq f}\right)\\
                &\scalemath{0.85}{:=g(h)\cdot \left(\bigcap\limits_{x\in R^{\mathbb{P}^*}_{\alpha}}(f(x)\rightarrow \{q\in \mathbb{P}: \mathcal{V}^{\mathbb{P}^*}_{\alpha} \Vdash_q \sim\sim\Diamond (x\in h)\})\cdot\bigcap\limits_{x\in R^{\mathbb{P}^*}_{\alpha}}(f(x)\leftarrow \{q\in \mathbb{P}: \mathcal{V}^{\mathbb{P}^*}_{\alpha} \Vdash_q \sim\sim\Diamond (x\in h)\})\right)}
            \end{align*}
                        
            and notice that, by Definition~\ref{definitionresiduatedkripkehierarchy} 2. c., 
            \begin{center}
                $\{p\in \mathbb{P}: \mathcal{V}_{\alpha+1}^{\mathbb{P}^*}\Vdash_p (f\in g)\}=\bigvee\limits_{h\in dom(g)}P_h$.
            \end{center}

            Furthermore, 
            
            \begin{align*}          
                &P_{f\subseteq h}=\bigcap\limits_{x\in R^{\mathbb{P}^*}_{\alpha}}(f(x)\rightarrow \{q\in \mathbb{P}: \mathcal{V}^{\mathbb{P}^*}_{\alpha} \Vdash_q \sim \sim\Diamond (x\in h)\})\\
                &=\bigcap\limits_{x\in R^{\mathbb{P}^*}_{\alpha}}(f(x)\rightarrow \sim\sim\gamma\{q\in \mathbb{P}: \mathcal{V}^{\mathbb{P}^*}_{\alpha} \Vdash_q  (x\in h)\}) \hspace{0.5cm} &&\mbox{\small(by Theorems }  \ref{gammadiamond} \mbox{ and } \ref{kripkecongruence})\\
                &=\bigcap\limits_{x\in R^{\mathbb{P}^*}_{\alpha}}(\sim\sim\gamma(f(x))\rightarrow \sim\sim\gamma\{q\in \mathbb{P}: \mathcal{V}^{\mathbb{P}^*}_{\alpha} \Vdash_q  (x\in h)\}) \hspace{0.5cm} &&\mbox{(since $f$ is $\gamma$-regular) }\\
                &=\bigcap\limits_{x\in R^{\mathbb{P}^*}_{\alpha}}\sim\sim(\sim\sim\gamma(f(x))\rightarrow \sim\sim\gamma\{q\in \mathbb{P}: \mathcal{V}^{\mathbb{P}^*}_{\alpha} \Vdash_q  (x\in h)\}) \hspace{0.5cm} &&\mbox{(by Example \ref{doublenegationisquanticnucleus})}\\
                &=\sim\bigvee\limits_{x\in R^{\mathbb{P}^*}_{\alpha}}\sim(\sim\sim\gamma(f(x))\rightarrow \sim\sim\gamma\{q\in \mathbb{P}: \mathcal{V}^{\mathbb{P}^*}_{\alpha} \Vdash_q  (x\in h)\}) \hspace{0.5cm} &&\mbox{(by Theorem \ref{cuantales1.5} item 4.)}
            \end{align*}
            
            Thus,
            
            \begingroup
            \addtolength{\jot}{1em}    
            \begin{align*} 
                &|P_{f\subseteq h}|=\left|\bigcap\limits_{x\in R^{\mathbb{P}^*}_{\alpha}}(f(x)\rightarrow \{q\in \mathbb{P}: \mathcal{V}^{\mathbb{P}^*}_{\alpha} \Vdash_q \sim \sim\Diamond (x\in h)\})\right|\\ &=\left|\sim\bigvee\limits_{x\in R^{\mathbb{P}^*}_{\alpha}}\sim(\sim\sim\gamma(f(x))\rightarrow \sim\sim\gamma\{q\in \mathbb{P}: \mathcal{V}^{\mathbb{P}^*}_{\alpha} \Vdash_q  (x\in h)\})\right|\\ 
                &=\lnot\bigcup\limits_{x\in R^{\mathbb{P}^*}_{\alpha}}\lnot(\lnot\lnot|\gamma(f(x))|\rightarrow \lnot\lnot|\{q\in \mathbb{P}: \mathcal{V}^{\mathbb{P}^*}_{\alpha} \Vdash_q  (x\in h)\}|)\hspace{2cm}\mbox{(by Remark }\ref{RemarkoperationsonH}) \\  
                &=\bigcap\limits_{x\in R^{\mathbb{P}^*}_{\alpha}}\lnot\lnot(\lnot\lnot|\gamma(f(x))|\rightarrow \lnot\lnot|\{q\in \mathbb{P}: \mathcal{V}^{\mathbb{P}^*}_{\alpha} \Vdash_q  (x\in h)\}|) \hspace{.5cm} \mbox{(by Theorem \ref{cuantales1.5} item 4.)}                  \\ 
                &=\bigcap\limits_{x\in R^{\mathbb{P}^*}_{\alpha}}\lnot\lnot|\gamma(f(x))|\rightarrow \lnot\lnot|\{q\in \mathbb{P}: \mathcal{V}^{\mathbb{P}^*}_{\alpha} \Vdash_q  (x\in h)\}| \hspace{2.5cm} \mbox{(by Example \ref{doublenegationisquanticnucleus})} \\
                &=\bigcap\limits_{x\in R^{\mathbb{P}^*}_{\alpha}}|\sim\sim\gamma(f(x))|\rightarrow \lnot\lnot|\{q\in \mathbb{P}: \mathcal{V}^{\mathbb{P}^*}_{\alpha} \Vdash_q  (x\in h)\}| \hspace{2.4cm} \mbox{(by definition of $\lnot$)} 
                \\
                &=\bigcap\limits_{x\in R^{\mathbb{P}^*}_{\alpha}}|f(x)|\rightarrow \lnot\lnot|\{q\in \mathbb{P}: \mathcal{V}^{\mathbb{P}^*}_{\alpha} \Vdash_q  (x\in h)\}| \hspace{4.6cm} \mbox{($f$ is $\gamma$-regular)}\\
                &=\bigcap\limits_{x\in R^{\mathbb{P}^*}_{\alpha}}f'(x')\rightarrow \lnot\lnot\ldbrack  x'\in h'\rdbrack_{\alpha}^\mathbb{H} \hspace{0.5cm} \mbox{(by the induction hypothesis 1. and the definition of $f'$) }\\
                &=\bigcap\limits_{x\in R^{\mathbb{P}^*}_{\alpha}}f'(x')\rightarrow \ldbrack  \lnot\lnot (x'\in h')\rdbrack_{\alpha}^\mathbb{H} \hspace{6.2cm} \mbox{(by definition of $\lnot$) }
            \end{align*}
            \endgroup
            
            In a similar way we can prove that
            \begin{center}
                $\scalemath{0.95}{|P_{h\subseteq f}|=\left|\bigcap\limits_{x\in R^{\mathbb{P}^*}_{\alpha}}(f(x)\leftarrow \{q\in \mathbb{P}: \mathcal{V}^{\mathbb{P}^*}_{\alpha} \Vdash_q \sim \sim\Diamond (x\in h)\})\right|=\bigcap\limits_{x\in R^{\mathbb{P}^*}_{\alpha}}f'(x')\leftarrow \ldbrack  \lnot\lnot (x'\in h')\rdbrack_{\alpha}^\mathbb{H}}$
            \end{center}
            
            Thus,

            \begin{align*}            
                \scalemath{0.9}{|P_h|}&=\left|g(h)\cdot \left(P_{f\subseteq h}\cdot P_{h\subseteq f}\right)\right| =|g(h)|\cdot (|P_{f\subseteq h}|\cdot |P_{h\subseteq f}|)\hspace{3cm} \mbox{ (by Remark \ref{RemarkoperationsonH}) }\\
                &\scalemath{0.8}{=|g(h)|\cdot \left(\bigcap\limits_{x\in R^{\mathbb{P}^*}_{\alpha}}f'(x')\rightarrow \ldbrack   \lnot\lnot (x'\in h')\rdbrack_{\alpha}^\mathbb{H}\cdot \bigcap\limits_{x\in R^{\mathbb{P}^*}_{\alpha}}f'(x')\leftarrow \ldbrack   \lnot\lnot (x'\in h')\rdbrack_{\alpha}^\mathbb{H}\right) \hspace{0.3cm} \mbox{ (by the previous computations)}}\\
                &\scalemath{0.85}{=|g(h)|\cap \left(\bigcap\limits_{x\in R^{\mathbb{P}^*}_{\alpha}}f'(x')\rightarrow \ldbrack   \lnot\lnot (x'\in h')\rdbrack_{\alpha}^\mathbb{H}\cap \bigcap\limits_{x\in R^{\mathbb{P}^*}_{\alpha}}f'(x')\leftarrow \ldbrack   \lnot\lnot (x'\in h')\rdbrack_{\alpha}^\mathbb{H}\right) \hspace{0.3cm} \mbox{ (since $\cap=\cdot$ in  $\mathbb{P}^*/\mathcal{F}_{\gamma}$)}}\\
                &\scalemath{0.9}{=g'(h')\cap \bigcap\limits_{x\in R_{\alpha}^{\mathbb{P}^*}}(f'(x')\leftrightarrow \ldbrack   \lnot\lnot (x'\in h')\rdbrack_{\alpha}^\mathbb{H}) \hspace{3.2cm} \mbox{ (properties of $\land$ and definition of $g'$)}}
            \end{align*}
            
            Therefore, 
            
            \begin{align*}
                \ldbrack  f'\in g'\rdbrack_{\alpha+1}^\mathbb{H}&=\bigcup\limits_{h'\in dom(g')} (g'(h')\cap \bigcap\limits_{x\in R_{\alpha}^{\mathbb{P}^*}}(f'(x')\leftrightarrow \ldbrack   \lnot\lnot (x'\in h')\rdbrack_{\alpha}^\mathbb{H})\hspace{0.5cm} &&\mbox{ (by definition of $\ldbrack  \cdot\in \cdot\rdbrack_{\alpha+1}^\mathbb{H}$)}\\
                &=\bigcup\limits_{h'\in dom(g')}|P_h|\\
                &=\left| \ \bigvee\limits_{h\in dom(g)}P_h \ \right|\hspace{0.5cm} &&\mbox{(by Theorem \ref{unionclass})}\\
                &=|\{p\in \mathbb{P}: \mathcal{V}_{\alpha+1}^{\mathbb{P}^*}\Vdash_p (f\in g)\}|&&\mbox{by Definition~\ref{definitionresiduatedkripkehierarchy} 2. c.}\\
            \end{align*}
        \end{enumerate}

        Now we have the result for atomic $MR$-formulas. It is straightforward to prove the result for the rest of $MR$-formulas by induction on formulas and by using Theorems \ref{kripkecongruence} and \ref{gammadiamond}. For this reason, we show it only for the product and the existential quantifier: Assume that $\varphi(x_1, ..., x_n)$ and $\psi(x_1, ..., x_n)$ are $MR-\mathcal{L}_{\in}$-formulas such that for all $a_1, ..., a_n\in R_{\alpha}^{\mathbb{P}^*}$
        
        \begin{center}
            $|\{p\in \mathbb{P}: \mathcal{V}^{\mathbb{P}^*}_{\alpha+1}\Vdash_p \varphi(a_1, ..., a_n)\}|=\ldbrack \varphi(a'_1, ..., a'_n)\rdbrack_{\alpha+1}^\mathbb{H}$ and $|\{p\in \mathbb{P}: \mathcal{V}^{\mathbb{P}^*}_{\alpha+1}\Vdash_p \psi(a_1, ..., a_n)\}|=\ldbrack \psi(a'_1, ..., a'_n)\rdbrack_{\alpha+1}^\mathbb{H}$  (induction hypothesis 2).
        \end{center}
        
        Then,        
        \begin{align*}
            &\ldbrack (\varphi\&\psi)(a'_1, ..., a'_n)\rdbrack_{\alpha+1}^\mathbb{H}=\ldbrack \varphi(a'_1, ..., a'_n)\rdbrack_{\alpha+1}^\mathbb{H}\cdot \ldbrack \psi(a'_1, ..., a'_n)\rdbrack_{\alpha+1}^\mathbb{H}\\
            &=|\{p\in \mathbb{P}: \mathcal{V}^{\mathbb{P}^*}_{\alpha+1}\Vdash_p \varphi(a_1, ..., a_n)\}|\cdot |\{p\in \mathbb{P}: \mathcal{V}^{\mathbb{P}^*}_{\alpha+1}\Vdash_p \psi(a_1, ..., a_n)\}| \hspace{0.3cm} &&\mbox{ \small(by induction hypothesis 2)} \\
            &=|\{p\in \mathbb{P}: \mathcal{V}^{\mathbb{P}^*}_{\alpha+1}\Vdash_p (\varphi\&\psi)(a_1, ..., a_n)\}| \hspace{0.5cm} &&\mbox{ (by Theorem \ref{kripkecongruence}) and}\\ & &&\mbox{Remark~\ref{RemarkoperationsonH})} 
        \end{align*}

        Take an $MR-\mathcal{L}_{\in}$-formula $\varphi(x_1, ..., x_n)$ and $a_1, ..., a_n\in R_{\alpha+1}^{\mathbb{P}^*}$.
        Assume that for every $a\in R_{\alpha+1}^{\mathbb{P}^*}$, $\varphi(a_1, ..., a_n, a)$ satisfies
        \begin{center}
            $|\{p\in \mathbb{P}: \mathcal{V}^{\mathbb{P}^*}_{\alpha+1}\Vdash_p \varphi(a_1, ..., a_n, a)\}|=\ldbrack (\varphi(a'_1, ..., a'_n, a')\rdbrack_{\alpha+1}^\mathbb{H}$ (induction hypothesis 3).
        \end{center}        
        Then,
        \begin{align*}
            \ldbrack \exists x\varphi(a'_1, ..., a'_n, x)\rdbrack_{\alpha+1}^\mathbb{H}&=\bigvee\limits_{a'\in R^{\mathbb{H}}_{\alpha+1}}\ldbrack (\varphi(a'_1, ..., a'_n, a')\rdbrack_{\alpha+1}^\mathbb{H}\hspace{0.7cm} \mbox{(by definition of $\ldbrack \exists x\varphi(a'_1, ..., a'_n, x)\rdbrack_{\alpha+1}^\mathbb{H}$)} \\
            &=\bigvee\limits_{a'\in R^{\mathbb{H}}_{\alpha+1}}|\{p\in \mathbb{P}: \mathcal{V}^{\mathbb{P}^*}_{\alpha+1}\Vdash_p \varphi(a_1, ..., a_n, a)\}| \hspace{0.5cm} \mbox{(by induction hypothesis 3)} \\
            &=|\{p\in \mathbb{P}: \mathcal{V}^{\mathbb{P}^*}_{\alpha+1}\Vdash_p  \exists x\varphi(a_1, ..., a_n, x)\}| \hspace{2.5cm} \mbox{(by Theorem \ref{kripkecongruence}) } 
        \end{align*}
     \\ \\ \indent    
        Now, let us prove that $g$ is extensional if and only if $g'$ is extensional.
    \\ \\         
        \indent If $g$ is extensional, then, for every $f, h\in R^{\mathbb{P}^*}_{\alpha}$, we have that
        \begin{center}
        $g(f)\cdot \{p\in \mathbb{P}: \mathcal{V}^{\mathbb{P}^*}_{\alpha}\Vdash_p (f=h)\}\subseteq g(h)$,
        \end{center}
        and by Theorem \ref{cuantales1} item 1., 
        \begin{center}
        $(g(f)\cdot \{p\in \mathbb{P}: \mathcal{V}^{\mathbb{P}^*}_{\alpha}\Vdash_p (f=h)\})\rightarrow g(h)=1$.
        \end{center}
        Thus, 
        \begin{center}
            $(g(f)\cdot \{p\in \mathbb{P}: \mathcal{V}^{\mathbb{P}^*}_{\alpha}\Vdash_p (f=h)\})\rightarrow g(h)\in \mathcal{F}_{\gamma}$. 
        \end{center}      
        By definition of $\leq$ on $\mathbb{P}^*/\mathcal{F}_{\gamma}$, we get
        \begin{center}
        $|g(f)\cdot \{p\in \mathbb{P}: \mathcal{V}^{\mathbb{P}^*}_{\alpha}\Vdash_p (f=h)\}|\leq |g(h)|$
        \end{center}
        by Remark \ref{RemarkoperationsonH}, this implies that
        \begin{center}
        $|g(f)|\land |\{p\in \mathbb{P}: \mathcal{V}^{\mathbb{P}^*}_{\alpha}\Vdash_p (f=h)\}|\leq |g(h)|$.
        \end{center}
        Recall that        
        \begin{center}
            $(g=h):=\Diamond\sim(\exists x)\sim (x\in g\rightarrow x\in h))\& (\Diamond\sim(\exists x)\sim(x\in h\leftarrow x\in g))$.
        \end{center} 
        Thus, by Theorems \ref{kripkecongruence} and \ref{gammadiamond},
        \begin{align*}            
	    &|\{p\in \mathbb{P}: \mathcal{V}^{\mathbb{P}^*}_{\alpha}\Vdash_p (f=h)\}|\\        &=|\{p\in \mathbb{P}: \mathcal{V}^{\mathbb{P}^*}_{\alpha}\Vdash_p (\Diamond\sim(\exists x)\sim (x\in g\rightarrow x\in h))\& (\Diamond\sim(\exists x)\sim(x\in h\leftarrow x\in g)))\}| \hspace{0.5cm} \mbox{} \\
            &=|\gamma\{p\in \mathbb{P}: \mathcal{V}^{\mathbb{P}^*}_{\alpha}\Vdash_p (\sim(\exists x)\sim (x\in g\rightarrow x\in h))\}|\cdot |\gamma\{p\in \mathbb{P}: \mathcal{V}^{\mathbb{P}^*}_{\alpha}\Vdash_p (\sim(\exists x)\sim (x\in g\leftarrow x\in h))\}| \hspace{0.5cm} \mbox{} 
        \end{align*}
        Notice that
        \begin{align*}
            &|\gamma\{p\in \mathbb{P}: \mathcal{V}^{\mathbb{P}^*}_{\alpha}\Vdash_p (\sim(\exists x)\sim (x\in g\rightarrow x\in h))\}|\\
            &=|\{p\in \mathbb{P}: \mathcal{V}^{\mathbb{P}^*}_{\alpha}\Vdash_p (\sim(\exists x)\sim (x\in g\rightarrow x\in h))\}|\hspace{1cm} &&\mbox{(by Corollary \ref{clasedegammaxesclasedex})}\\
	    &=\ldbrack \sim(\exists x)\sim (x\in g\rightarrow x\in h))\rdbrack_{\alpha}^\mathbb{H} \hspace{0.5cm} &&\mbox{(by induction hypothesis 1.)} 
        \end{align*}
        In a similar way, we prove that 
        \begin{center}
            $|\gamma\{p\in \mathbb{P}: \mathcal{V}^{\mathbb{P}^*}_{\alpha}\Vdash_p (\sim(\exists x)\sim (x\in g\leftarrow x\in h))\}|=\ldbrack \sim(\exists x)\sim (x\in g\leftarrow x\in h))\rdbrack_{\alpha}^\mathbb{H}$
        \end{center}
        Therefore, 
        \begin{center}
            $|\{p\in \mathbb{P}: \mathcal{V}^{\mathbb{P}^*}_{\alpha}\Vdash_p (f=h)\}|=\ldbrack \sim(\exists x)\sim (x\in g\rightarrow x\in h))\rdbrack_{\alpha}^\mathbb{H}\land \ldbrack\sim(\exists x)\sim(x\in h\leftarrow x\in g)\rdbrack_{\alpha}^\mathbb{H}$
        \end{center}
        Thus, by definition of $g'$
        \begin{center}
        $g'(f')\land \ldbrack \sim(\exists x)\sim (x\in g\rightarrow x\in h))\rdbrack_{\alpha}^\mathbb{H}\land \ldbrack\sim(\exists x)\sim(x\in h\leftarrow x\in g)\rdbrack_{\alpha}^\mathbb{H}\leq g'(h')$
        \end{center}
        which proves that $g'$ is extensional. 

        On the other hand, if $g'$ is extensional, by using (backwards) the argument given above, we get that
        \begin{center}
            $(g(f)\cdot \{p\in \mathbb{P}: \mathcal{V}^{\mathbb{P}^*}_{\alpha}\Vdash_p (f=h)\})\rightarrow g(h)\in \mathcal{F}_{\gamma}$. 
        \end{center}
        which means
        \begin{center}
        $\gamma((g(f)\cdot \{p\in \mathbb{P}: \mathcal{V}^{\mathbb{P}^*}_{\alpha}\Vdash_p (f=h)\})\rightarrow g(h))=1$
        \end{center}
        but, since $\gamma$ respects implications, we get
        \begin{center}
        $(g(f)\cdot \{p\in \mathbb{P}: \mathcal{V}^{\mathbb{P}^*}_{\alpha}\Vdash_p (f=h)\})\rightarrow \gamma(g(h))=1$.
        \end{center}
        Hence, by Theorem \ref{cuantales2} item 2.,  $\gamma(g(h))\leq (\sim\sim \gamma(g(h)))$ and by Theorem \ref{cuantales1} item 9., we have
        \begin{center}
        $(g(f)\cdot \{p\in \mathbb{P}: \mathcal{V}^{\mathbb{P}^*}_{\alpha}\Vdash_p (f=h)\})\rightarrow (\sim\sim\gamma(g(h)))=1$,
        \end{center}       
        but $g$ is a $\gamma$-regular function,
        \begin{center}
        $(g(f)\cdot \{p\in \mathbb{P}: \mathcal{V}^{\mathbb{P}^*}_{\alpha}\Vdash_p (f=h)\})\rightarrow g(h)=1$
        \end{center}
        which implies, by Theorem \ref{cuantales1} item 1., that 
        \begin{center}
            $g(f)\cdot \{p\in \mathbb{P}: \mathcal{V}^{\mathbb{P}^*}_{\alpha}\Vdash_p (f=h)\}\subseteq g(h)$
        \end{center}
        That is, $g$ is extensional.

    \end{proof}

    \begin{Coro}[\cite{Moncayo2023}, Corollary 3.2.30.]\label{phiiffdiamondphionmodels}
        If $\varphi$ is an $\mathcal{L}_{\in}$-sentence with no universal quantifiers, then $\varphi$ is valid in the model $R_{\alpha}^{\mathbb{H}}$ (that is $\ldbrack \varphi\rdbrack_{\alpha}^\mathbb{H}=1_{\mathbb{H}}$), if and only if, $\Diamond \varphi$ is valid in $(\mathbb{P}, \leq,  \delta, \Vdash, R_{\alpha}^{\mathbb{P}^*})$ (that is $R_{\alpha}^{\mathbb{P}^*}\Vdash_p \Diamond\varphi$ for every $p\in \mathbb{P}$.)
    \end{Coro}
    \begin{proof}
        We have that 
        \begin{align*}
            \ldbrack \varphi\rdbrack_{\alpha}^\mathbb{H}=1_{\mathbb{H}}&\mbox{ iff }\ldbrack \varphi\rdbrack_{\alpha}^\mathbb{H}=|\mathbb{P}| \hspace{0.5cm} &&\mbox{(by definition of $1_{\mathbb{H}}=1_{\mathbb{P}^*/\mathcal{F}_{\gamma}}$)} \\
            &\mbox{ iff }|\{p\in \mathbb{P}: \mathcal{V}^{\mathbb{P}^*}_{\alpha}\Vdash_p \varphi\}|=|\mathbb{P}|\hspace{0.5cm} &&\mbox{(by Theorem \ref{isomorphismresiduatedclassic})} \\
            & \mbox{ iff }\gamma\{p\in \mathbb{P}: \mathcal{V}^{\mathbb{P}^*}_{\alpha}\Vdash_p \varphi\}=\gamma (\mathbb{P}) \hspace{0.5cm} &&\mbox{(by  Theorem \ref{ABiffjAJB})} \\
            & \mbox{ iff }\{p\in \mathbb{P}: \mathcal{V}^{\mathbb{P}^*}_{\alpha}\Vdash_p \Diamond\varphi\}]=\mathbb{P}  \hspace{0.5cm} &&\mbox{(by Theorem \ref{gammadiamond} and since $\gamma$ is expansive)} 
        \end{align*}
    \end{proof}
        
\bibliographystyle{amsalpha}

\end{document}